\documentclass[11pt,reqno]{amsart}
\usepackage[letterpaper,margin=1in,footskip=0.25in]{geometry}
\usepackage{microtype}
\usepackage{mathtools}
\usepackage{enumitem}
\usepackage[noadjust]{cite}
\renewcommand{\citepunct}{;\ }

\usepackage{trimclip}
\PassOptionsToPackage{dvipsnames}{xcolor}
\usepackage{tikz-cd}

\PassOptionsToPackage{pdfusetitle,colorlinks,pagebackref,bookmarksdepth=3}{hyperref}
\usepackage{bookmark}
\hypersetup{citecolor=[HTML]{2E7E2A},linkcolor=[HTML]{800006},urlcolor=[HTML]{8A0087}}
\usepackage[hyphenbreaks]{breakurl}
\newtheorem{theorem}{Theorem}[section]
\newtheorem{corollary}[theorem]{Corollary}
\newtheorem{question}[theorem]{Question}
\newtheorem{proposition}[theorem]{Proposition}

\newtheorem{alphthm}{Theorem}

\newtheorem{alphcor}[alphthm]{Corollary}

\theoremstyle{definition}
\newtheorem{definition}[theorem]{Definition}
\newtheorem{setup}[theorem]{Setup}

\theoremstyle{remark}
\newtheorem{remark}[theorem]{Remark}

\newtheoremstyle{cited}{.5\baselineskip\@plus.2\baselineskip\@minus.2\baselineskip}{.5\baselineskip\@plus.2\baselineskip\@minus.2\baselineskip}{\itshape}{}{\bfseries}{\bfseries .}{5pt plus 1pt minus 1pt}{\thmname{#1}\thmnumber{ #2}\thmnote{ \normalfont#3}}
\theoremstyle{cited}
\newtheorem{citedthm}[theorem]{Theorem}

\newtheoremstyle{citeddef}{.5\baselineskip\@plus.2\baselineskip\@minus.2\baselineskip}{.5\baselineskip\@plus.2\baselineskip\@minus.2\baselineskip}{}{}{\bfseries}{\bfseries .}{5pt plus 1pt minus 1pt}{\thmname{#1}\thmnumber{ #2}\thmnote{ \normalfont#3}}
\theoremstyle{citeddef}
\newtheorem{citeddef}[theorem]{Definition}

\DeclareMathOperator{\Spec}{Spec}

\newcommand{\CC}{\mathbb{C}}
\newcommand{\NN}{\mathbb{N}}
\newcommand{\QQ}{\mathbb{Q}}
\newcommand{\cO}{\mathcal{O}}
\newcommand{\fm}{\mathfrak{m}}

\newcommand{\redpair}[1]{#1, \Sigma_{#1}}

\DeclarePairedDelimiterX\Set[1]\{\}{#1}

\newcommand{\hooklongrightarrow}{\lhook\joinrel\longrightarrow}

\makeatletter
\newcommand{\longtwoheadrightarrow}{\mathrel{\text{\longtwo@rightarrow}}}
\newcommand{\two@rightarrow}{%
  \sbox0{$\m@th\rightarrow$}%
  \smash{\rlap{\kern0.1\wd0 \clipbox{{.3\width} {-\height} 0pt {-\height}}{$\m@th\rightarrow$}}}%
  $\m@th\rightarrow$%
}
\newcommand{\longtwo@rightarrow}{%
  \sbox0{$\m@th\longrightarrow$}%
  \smash{\rlap{\kern0.175\wd0 \clipbox{{.3\width} {-\height} 0pt {-\height}}{$\m@th\longrightarrow$}}}%
  $\m@th\longrightarrow$%
}
\makeatother

\begin{document}
\title[Pure subrings of Du Bois singularities]{Pure subrings of
Du Bois singularities\\are Du Bois singularities}
\author{Charles Godfrey}
\address{AKASA\\South San Francisco, CA 94080\\USA}
\email{\href{mailto:godfrey.cw@gmail.com}{godfrey.cw@gmail.com}}
\urladdr{\url{https://godfrey-cw.github.io/}}
\thanks{The
first author was partially supported by the University of Washington Department
of Mathematics Graduate Research Fellowship, and by the NSF grant DMS-1440140,
administered by the Mathematical Sciences  Research Institute, while in
residence  at  MSRI during the program Birational Geometry and Moduli Spaces.}
\author{Takumi Murayama}
\address{Department of Mathematics\\Purdue University\\West Lafayette, IN
47907-2067\\USA}
\email{\href{mailto:murayama@purdue.edu}{murayama@purdue.edu}}
\urladdr{\url{https://www.math.purdue.edu/~murayama/}}
\thanks{The second author was
supported by the National Science Foundation under Grant Nos.\ DMS-1701622,
DMS-1902616, and DMS-2201251.}



\makeatletter
  \hypersetup{
    pdfauthor={Charles Godfrey and Takumi Murayama}
    pdftitle={Pure subrings of Du Bois singularities are Du Bois singularities},
    pdfsubject=\@subjclass,pdfkeywords=\@keywords
  }
\makeatother

\begin{abstract}
  Let \(R \to S\) be a cyclically pure map of Noetherian \(\mathbb{Q}\)-algebras.
  In this paper, we show that if \(S\) has Du Bois singularities, then \(R\)
  has Du Bois singularities.
  Our result is new even when \(R \to S\) is faithfully flat.
  Our proof also yields interesting results in prime characteristic and in mixed
  characteristic.
  As a consequence, we show that if \(R \to S\) is a cyclically pure map of
  rings essentially of finite type over the complex numbers \(\mathbb{C}\),
  \(S\) has log canonical type singularities,
  and \(K_R\) is Cartier, then \(R\) has log canonical singularities.
  Along the way, we prove a version of the
  key injectivity theorem of Kov\'acs and Schwede for
  Noetherian schemes of equal characteristic zero that have isolated
  non-Du Bois points.
  Throughout the paper, we use the characterization of the complex
  \(\underline{\Omega}^0_X\) and of Du Bois singularities in terms of
  sheafification with respect to Grothendieck topologies.
\end{abstract}

\maketitle 

\section{Introduction}
\subsection{Background}
Let \(R \to S\) be a \textsl{cyclically pure}
map of rings, which we recall means that
\(IS \cap R = I\) for every ideal \(I \subseteq R\) \cite[p.\
463]{Hoc77}.
Cyclically pure maps seem to have been first studied by
Besserre \cite{Bes62,Bes67}, who calls them \textsl{good} ring maps.
Examples of cyclically pure maps include inclusions of rings of invariants by
linearly reductive group actions, split maps, and faithfully flat maps.
By \cite{HR74,Kem79,HH95,HM18}, if \(S\) is regular, then \(R\) is
Cohen--Macaulay.\smallskip
\par An interesting question arising from \cite{HR74} and these subsequent results
is the following:
\begin{question}\label{ques:cycpure}
  If \(R \to S\) is a cyclically pure map of rings, what 
  properties descend from \(S\) to \(R\)?
\end{question}
Here we say a property \(\mathbf{P} \) \textsl{descends} from \(S\) to \(R\) if whenever
\(S\) satisfies \(\mathbf{P}\), so does \(R\). Note that the preceding example is not precisely
a descending property since the property on \(S \) is regularity
whereas the property on \(R\) is Cohen--Macaulayness (and
neither regularity nor Cohen--Macaulayness descend along
cyclically pure maps in general \cite[\S2]{HR74}). On the other hand, Noetherianity
and normality descend from \(S\) to \(R\)
\cite[Proposition 6.15]{HR74}.\smallskip
\par Question \ref{ques:cycpure} has attracted particular attention for different
classes of singularities.
For example:
\begin{enumerate}[label=(\textup{\Roman*})]
  \item Boutot showed that if \(R\) and \(S\) are essentially of finite type
    over a field of characteristic zero
    and \(S\) has rational singularities, then \(R\)
    has rational singularities \cite[Th\'eor\`eme on p.\ 65]{Bou87}.
    Using techniques from this paper, the second author showed Boutot's result
    holds more generally for Noetherian \(\QQ\)-algebras \cite[Theorem
    C]{MurKV}. For other extensions of Boutot's result, see
    \cite{Smi97,Sch08,HM18}.
  \item Recently, Zhuang showed that
    if \(R\) and \(S\) are essentially of finite type over an algebraically
    closed field of characteristic zero, and \(S\) has klt type singularities,
    then \(R\) has klt type singularities \cite[Theorem 1.1]{Zhuang}.
    In the appendix to \cite{Zhuang},
    Lyu extended Zhuang's result to maps of excellent \(\QQ\)-algebras with
    dualizing complexes \cite[Theorem A.1]{Zhuang}.
    These results extend results from \cite{Kaw84,Sch05,BGLM}.
\end{enumerate}
\subsection{Main results}
\par In this paper, we answer Question \ref{ques:cycpure} for Du Bois
singularities by proving the following analogue of Boutot's theorem.
\begin{alphthm}\label{thm:boutotdb}
  Let \(R \to S\) be a cyclically pure map of Noetherian
  \(\mathbb{Q}\)-algebras.
  If \(S\) has Du Bois singularities with respect to the h topology, then \(R\)
  has Du Bois singularities with respect to the h topology.
\end{alphthm}
See Definition \ref{def:dbviatop} for our definition of Du Bois singularities,
which coincides with the usual notion for rings essentially of finite type over
\(\mathbb{C}\) \cite[(3.5)]{Ste83} and for quasi-excellent
\(\mathbb{Q}\)-algebras \cite[Definition 7.3.1]{Murinj}.
See Proposition \ref{prop:sameashyperresdef}.
For quasi-excellent \(\QQ\)-algebras, our definition is also equivalent to the
characteristic-free definition of Du Bois singularities due to Huber and Kelly
\cite[Definition 7.19]{Hub16}.
In fact, our proof of Theorem \ref{thm:boutotdb} yields interesting results in
prime characteristic and even in mixed characteristic.
See Theorem \ref{cor:dbdescends}.
In prime characteristic, we recover results on descent of \(F\)-injectivity
under faithfully flat, quasi-finite and pure, or strongly pure maps due to Datta
and the second author \cite[Theorem 3.8 and Proposition 3.9]{DM}.\smallskip
\par Theorem \ref{thm:boutotdb} is new even when \(R \to S\) is faithfully flat.
For rings essentially of finite type over \(\CC\),
the case when \(R \to S\) splits is due to Kov\'acs \cite[Corollary 2.4]{Kov99},
who also proves that Du Bois singularities descend along morphisms \(f\colon Y
\to X\) of schemes when \(\cO_X \to \mathcal{R} f_*\cO_Y\) splits in the derived
category.
In Theorem \ref{cor:dbdescends},
we extend Kov\'acs's result to morphisms for which
\(\cO_X \to \mathcal{R} f_*\cO_Y\) induces injective maps on local cohomology,
and also prove a
version for Du Bois pairs.
\par Theorem \ref{thm:boutotdb} contrasts with the situation in
prime characteristic.
The prime characteristic analogue of Du Bois singularities is the notion of
\(F\)-injectivity \cite{Sch09}.
Question \ref{ques:cycpure} has a negative answer for \(F\)-injectivity
\cite[Example 3.3(1)]{Wat97}, even if the map \(R
\to S\) splits as a ring map \cite[Example 6.6 and Remark 6.7]{Ngu12}.
However, as mentioned above,
Question \ref{ques:cycpure} does hold for \(F\)-injectivity under the
additional hypothesis that \(R \to S\) is either quasi-finite or strongly pure
\cite[Theorem 3.8 and Proposition 3.9]{DM}.\smallskip
\par As a consequence of Theorem \ref{thm:boutotdb}, we prove the following special
cases of a question of
Zhuang \cite[Question 2.10]{Zhuang} (cf.\ \citeleft\citen{Sch05}\citemid
(3.14)\citepunct \citen{BGLM}\citemid Question 8.5\citeright), who asked
whether Question \ref{ques:cycpure} holds for log canonical type singularities.
Statement \((\ref{cor:boutotlcqgorqcartier})\) below is the analogue of
\cite[Lemma 2.3]{Zhuang} for log canonical type singularities.
We say that a normal complex variety \(X\)
with canonical divisor \(K_X\) has \textsl{log
canonical type singularities} if for some effective \(\QQ\)-divisor \(\Delta\),
the pair \((X,\Delta)\) is log canonical.
\begin{alphcor}\label{cor:boutotlcqgor}
  Let \(R \to S\) be a cyclically pure map of essentially of finite type
  \(\CC\)-algebras.
  \begin{enumerate}[label=\((\roman*)\),ref=\roman*]
    \item\label{cor:boutotlcqgorcartier}
      If \(S\) is normal and has Du Bois singularities (for example, if
      \(S\) has log canonical type singularities) and \(K_R\) is Cartier,
      then \(R\) has log canonical singularities.
    \item\label{cor:boutotlcqgorqcartier}
      Let \(f\colon Y \to X\) denote the morphism of affine schemes
      corresponding to \(R \to S\).
      Let \(U\) be the Cartier locus of \(K_X\).
      If \(Y\) has log canonical type singularities, \(K_X\) is
      \(\QQ\)-Cartier, and \(Y \setminus f^{-1}(U)\) has codimension at least
      two in \(Y\), then \(X\) has log canonical singularities.
  \end{enumerate}
\end{alphcor}
\subsection{Outline}
To prove Theorem \ref{thm:boutotdb}, we develop
a characteristic-free notion of
Du Bois singularities that is essentially due to Huber and Kelly
\cite[Definition 7.19]{Hub16} in Section \ref{sect:dbgrotop}.
The key idea behind their definition is that although the usual definition of
\(\underline{\Omega}^0_X\) appearing in the definition of Du Bois singularities
does not make sense without the existence of resolutions of singularities, an
alternative description of \(\underline{\Omega}^0_X\) as the derived pushforward
of the sheafification of
\(\cO_X\) with respect to the h topology (in equal characteristic zero) can be
made in arbitrary characteristic.
This description is originally due to Lee \cite[Theorem 4.16]{Lee09} (see also
\cite[Corollary 6.16 and Theorem 7.12]{HJ14}).\medskip
\par In Section \ref{sect:dbinchar0}, we characterize Du Bois singularities in
terms of injectivity of maps on local cohomology modules in equal characteristic
zero (Theorem \ref{thm:splittingcriterion}).
The key ingredient is the following version of the key injectivity
theorem of Kov\'acs and Schwede 
\citeleft\citen{KS16deforms}\citemid Theorem 3.3\citepunct \citen{KS16}\citemid 
Theorem 3.2\citepunct \citen{MSS17}\citemid Lemma 3.2\citepunct
\citen{KK20}\citemid Theorem 6.3 and Corollary 6.5\citeright, which is currently
known for schemes essentially of finite type over a field of characteristic
zero.
\par We prove a version of their key injectivity theorem for Noetherian schemes of
equal characteristic zero that have isolated non-Du Bois points.
For varieties, this special case of the key injectivity theorem is also due to
Kov\'acs \citeleft\citen{Kov99}\citemid Proof of Lemma 2.2\citepunct
\citen{Kov00}\citemid Lemma 1.4\citeright\ and Schwede \cite[Proposition
5.11]{Sch09}.
The reason it is called an injectivity theorem is that the usual statement for
complex varieties is the Matlis dual of the surjectivity statement below.
\begin{alphthm}\label{thm:keyinjintro}
  Let \(X\) be a separated Noetherian scheme of equal characteristic zero.
  Let \(x \in X\) be a point such that \(\Spec(\cO_{X,x}) \setminus
  \{x\}\) is Du Bois
  with respect to the h topology.
  Then, for every \(i\), the natural morphism
  \[
    H^i_x\bigl(\Spec(\cO_{X,x}),\mathcal{O}_{X,x}\bigr)
    \longrightarrow \mathbb{H}^i_x\bigl(\Spec(\cO_{X,x}),
    \underline{\Omega}^0_{X,\mathrm{h},x}\bigr).
  \]
  is surjective.
\end{alphthm}
In fact, we prove a more general version of Theorem \ref{thm:keyinjintro} for
pairs.
See Theorem \ref{thm:keyinj} for the full statement.
For the proof, we construct an absolute version of the Nagata compactification of
\(X\) using the notion of the Zariski--Riemann space associated to a pair
\cite{FK06,Tem10,FK18} (see Definition \ref{def:zrforpairs}).
This compactification 
is a locally ringed space \(\langle X \rangle_{\mathrm{cpt}}\) that contains
\(X\) as a subspace whose complement is ind-constructible, and is constructed as
the inverse limit of proper schemes over
\(\mathbb{Q}\).
Using the long exact sequence on local cohomology, we reduce Theorem
\ref{thm:keyinjintro} to a surjectivity statement for cohomology on
\(\langle X \rangle_{\mathrm{cpt}}\).
This last surjectivity statement is a consequence of the degeneration of the
Hodge-to-de Rham spectral sequence \citeleft\citen{Del74}\citemid
\S8.2\citepunct \citen{DB81}\citemid Th\'eor\`eme 4.5\((iii)\)\citeright\ 
and Grothendieck's limit theorem for inverse limits of
toposes \cite[Expos\'e VI, Corollaire 8.7.7]{SGA42}.\smallskip
\par Finally, in Section \ref{sect:mainresultproofs}, we prove Theorem
\ref{thm:boutotdb}, its generalization to pairs
(Theorem \ref{cor:dbdescends}), and
Corollary \ref{cor:boutotlcqgor}.

\subsection*{Conventions}
All rings are commutative with identity, and all ring maps are unital.
A map \(\varphi\colon M \to M'\) of \(R\)-modules over a ring \(R\) is \textsl{pure} if the
base change
\[
  \mathrm{id} \otimes_R \varphi\colon
  N \otimes_R M \longrightarrow N \otimes_R M'
\]
is injective for every
\(R\)-module \(N\) \citeleft\citen{Coh59}\citemid p.\ 383\citepunct
\citen{Oli70}\citemid D\'efinition 1.1\citeright.

\subsection*{Acknowledgments}
We would like to thank
Donu Arapura,
Rankeya Datta,
J\'anos Koll\'ar,
S\'andor J. Kov\'acs,
Linquan Ma,
Kenji Matsuki,
Joaqu\'in Moraga,
Karl Schwede, and
Farrah Yhee
for helpful discussions.
We also thank the referee for their helpful comments.

\section{Du Bois pairs via Grothendieck topologies}\label{sect:dbgrotop}
\subsection{Grothendieck topologies}
To define Du Bois pairs in full generality, we will need to work with the rh,
cdh, eh, or h topologies.
We define these Grothendieck topologies and other topologies that appear in
their definitions following \cite[Definition 2.5]{GK15}.
\begin{definition}\label{def:grotops}
  Let \(S\) be a separated Noetherian scheme.
  We consider the following topologies on a small category of
  schemes that are of finite type over \(S\), constructed as in
  \cite[\href{https://stacks.math.columbia.edu/tag/020M}{Tag
  020M}]{stacks-project}.
  We will consider finite families
  \begin{equation}\label{eq:gk155}
    \Set[\big]{f_i \colon U_i \longrightarrow X}_{i=1}^n
  \end{equation}
  of morphisms in \(\mathrm{Sch}/S\) in the definitions below.
  \begin{enumerate}[label=\((\roman*)\),ref=\roman*]
    \item The \textsl{Zariski topology}, denoted by \(\mathrm{Zar}\), is
      generated by families \eqref{eq:gk155} which are jointly
      surjective and where each \(f_i\) is an open immersion.
      We allow \(n = 0\) so that the empty family is a covering family for
      the empty scheme.
    \item \cite[\S1.1]{Nis89}
      The \textsl{Nisnevich topology}, denoted by \(\mathrm{Nis}\), is
      generated by \textsl{completely decomposed} families \eqref{eq:gk155}
      where each \(f_i\) is \'etale.
      Here, a family \eqref{eq:gk155} is \textsl{completely decomposed} if for
      every \(x \in X\) there are an index \(i\) and a point \(u \in U_i\) such
      that \(f_i(u) = x\) and \([k(u) : k(x)] = 1\).
    \item The \textsl{\'etale topology}, denoted by \(\mathrm{\acute{E}t}\), is
      generated by families \eqref{eq:gk155} which are jointly
      surjective and where each \(f_i\) is a \'etale.
    \item \cite[Definition 18.3]{Ful98}
      The \textsl{cdp topology}, denoted by \(\mathrm{cdp}\), is generated
      by completely decomposed families \eqref{eq:gk155} where each \(f_i\) is
      proper.
      Such a family where \(n = 1\) is called a \textsl{cdp morphism}
      \cite[Definition 2.10]{HKK17}.
      Note that cdp morphisms are called \textsl{envelopes} in \cite[Definition
      18.3]{Ful98} and \textsl{proper cdh covers} in \cite[Definition 5.7]{SV00}.
    \item The \textsl{finite topology}, denoted by \(\mathrm{f}\), is generated
      by families \eqref{eq:gk155} which are jointly surjective and
      where each \(f_i\) is finite.
    \item The \textsl{proper topology}, denoted by \(\mathrm{prop}\), is
      generated by families \eqref{eq:gk155} which are jointly surjective and
      where each \(f_i\) is proper.
    \item \cite[Definition 1.2]{GL01}
      The \textsl{rh topology}, denoted by \(\mathrm{rh}\), is the coarsest
      topology finer than both \(\mathrm{Zar}\) and \(\mathrm{cdp}\).
    \item \cite[Definition 5.7]{SV00}
      The \textsl{cdh topology}, denoted by \(\mathrm{cdh}\), is the coarsest
      topology finer than both \(\mathrm{Nis}\) and \(\mathrm{rh}\).
    \item \cite[Definition 2.1]{Gei06}
      The \textsl{eh topology}, denoted by \(\mathrm{eh}\), is the coarsest
      topology finer than both \(\mathrm{\acute{E}t}\) and \(\mathrm{rh}\).
    \item \cite[Definition 3.1.2 and Lemma 3.4.2]{Voe96}
      The \textsl{qfh topology}, denoted by \(\mathrm{qfh}\), is the coarsest
      topology finer than both \(\mathrm{\acute{E}t}\) and \(\mathrm{f}\).
    \item \cite[Definition 6.2]{HKK17}
      The \textsl{sdh topology}, denoted by \(\mathrm{sdh}\), is the coarsest
      topology finer than both \(\mathrm{\acute{E}t}\) and the topology
      generated by \textsl{separably decomposed} families \eqref{eq:gk155}.
      Here, a family \eqref{eq:gk155} is \textsl{separably decomposed} if
      each \(f_i\) is proper and if, for
      every \(x \in X\), there are an index \(i\) and a point \(u \in U_i\) such
      that \(f_i(u) = x\) and \(k(u)/k(x)\) is finite separable.
    \item\label{def:grotophtop}
      \citeleft\citen{Voe96}\citemid Definition 3.1.2\citepunct
      \citen{GL01}\citemid Definition 1.1 and Theorem 4.1\citeright\ 
      The \textsl{h topology}, denoted by \(\mathrm{h}\), is the coarsest
      topology finer than both \(\mathrm{Zar}\) and \(\mathrm{prop}\).
      Note that this is what \cite{GL01} and
      \cite[\href{https://stacks.math.columbia.edu/tag/0DBC}{Tag
      0DBC} and \href{https://stacks.math.columbia.edu/tag/0ETQ}{Tag 0ETQ}]{stacks-project}
      call the \textsl{ph topology}.
  \end{enumerate}
  These topologies are related in the following manner, where the arrows below
  point to finer topologies
  \citeleft\citen{GK15}\citemid (6)\citepunct \citen{EM19}\citemid p.\
  5288\citeright:
  \[
    \begin{tikzcd}
      & & & & \mathrm{f} \rar \dar & \mathrm{prop} \arrow{dd}\\
      & \mathrm{Zar} \rar \dar & \mathrm{Nis} \rar \dar & \mathrm{\acute{E}t}
      \dar \rar & \mathrm{qfh} 
      \arrow[end anchor={[xshift=1pt,yshift=-1.5pt]north west}]{dr}\\
      \mathrm{cdp} \rar & \mathrm{rh} \rar & \mathrm{cdh} \rar & \mathrm{eh}
      \arrow{r} & \mathrm{sdh} \rar & \mathrm{h}\mathrlap{.}
    \end{tikzcd}
  \]
\end{definition}
\subsection{Du Bois pairs}
We define Du Bois pairs following
\citeleft\citen{Kov11}\citemid \S\S3.C--3.D\citepunct
\citen{KS16}\citemid \S2C\citeright,
adapted to working with Grothendieck
topologies instead of hyperresolutions.
\begin{definition}\label{def:genpair}
  A \textsl{(generalized) pair \((X,\Sigma)\)} consists of a
  separated Noetherian scheme together with a closed
  subscheme \(\Sigma \subseteq X\).
  A \textsl{morphism of pairs}
  \[
    f\colon \bigl(\redpair{X}\bigr) \longrightarrow \bigl(\redpair{Y}\bigr)
  \]
  is a morphism of schemes \(f\colon X \to Y\) such that
  \(f(\Sigma_X) \subseteq \Sigma_Y\).
  \par Let \(f\colon (\redpair{X}) \to
  (\redpair{Y})\) be a morphism of pairs.
  Letting \(\tau \in \{\mathrm{rh},\mathrm{cdh},\mathrm{eh},\mathrm{sdh},\mathrm{h}\}\),
  we have the commutative diagram
  \[
    \begin{tikzcd}[row sep={between origins,3em},column sep={between origins,4em}]
      & \Sigma_{Y,\tau} \arrow{rr} \arrow{dd}[near start]{\rho_{\Sigma_Y,\tau}}
      & & Y_\tau \arrow{dd}{\rho_{Y,\tau}}\\
      \Sigma_{X,\tau} \arrow{dd}[swap]{\rho_{\Sigma_X,\tau}}
      \arrow[crossing over]{rr} \arrow{ur}
      & & X_\tau \arrow{ur}[swap]{f_\tau}\\
      & \Sigma_Y \arrow[hook]{rr} & & Y\\
      \Sigma_X \arrow[hook]{rr} \arrow{ur} &
      & X \arrow[leftarrow,crossing over]{uu}[swap,near end]{\rho_{X,\tau}}
      \arrow{ur}[swap]{f}
    \end{tikzcd}
  \]
  of morphisms of sites, where the schemes in the bottom half of the
  diagram are thought of as their associated Zariski sites.
  This commutative diagram induces the commutative diagram
  \begin{equation}\label{eq:omegafunctorial}
    \begin{tikzcd}[row sep={between origins,3em},column sep={between origins,5.5em}]
      & \mathcal{R}f_*\mathcal{I}_{\Sigma_Y \subseteq Y} \arrow{rr} \arrow{dd}
      & & \mathcal{R}f_*\cO_Y \arrow{rr} \arrow{dd}
      & & \mathcal{R}f_*\cO_{\Sigma_Y} \arrow{dd} \arrow{rr}{+1} &[-3em] & {}\\
      \mathcal{I}_{\Sigma_X \subseteq X} \arrow{ur} \arrow{dd}
      \arrow[crossing over]{rr}
      & & \cO_X \arrow[crossing over]{rr} \arrow{ur}
      & & \cO_{\Sigma_X} \arrow{ur} \arrow[crossing over]{rr}{+1} & & {}\\
      & \mathcal{R}f_*\underline{\Omega}^0_{Y,\Sigma_Y,\tau} \arrow{rr}
      & & \mathcal{R}f_*\mathcal{R}\rho_{Y,\tau*}\mathcal{O}_{Y_\tau} \arrow{rr}
      & & \mathcal{R}f_*\mathcal{R}\rho_{\Sigma_Y,\tau*}\mathcal{O}_{\Sigma_{Y,\tau}}
      \arrow{rr}{+1} & & {}\\
      \underline{\Omega}^0_{X,\Sigma_X,\tau} \arrow{rr} \arrow{ur}
      & & \mathcal{R}\rho_{X,\tau*}\mathcal{O}_{X_\tau} \arrow{rr} \arrow{ur}
      \arrow[leftarrow,crossing over]{uu}
      & & \mathcal{R}\rho_{\Sigma_X,\tau*}\mathcal{O}_{\Sigma_{X,\tau}} \arrow{ur}
      \arrow[leftarrow,crossing over]{uu} \arrow{rr}{+1} & & {}
    \end{tikzcd}
  \end{equation}
  where each row is an exact triangle.
  Here, we set
  \begin{align*}
    \underline{\Omega}^0_{X,\Sigma_X,\tau} \coloneqq
    \mathcal{R}\rho_{X,\tau*}\mathcal{I}_{\Sigma_X,\tau}
    \cong{}&
    {\operatorname{Cone}\bigl( \mathcal{R}\rho_{X,\tau*}\mathcal{O}_{X_\tau}
      \longrightarrow
    \mathcal{R}\rho_{\Sigma,\tau*}\mathcal{O}_{\Sigma_{X,\tau}} \bigr)[-1]}
  \intertext{and similarly for \(Y\).
  We also set}
    \underline{\Omega}^0_{X,\tau} \coloneqq{}&
    \underline{\Omega}^0_{X,\emptyset,\tau} \coloneqq
    \mathcal{R}\rho_{X,\tau*}\mathcal{O}_{X_\tau}.
  \end{align*}
\end{definition}
\begin{remark}\label{rem:naturalisosehcdhrh}
  Let \(X\) be a separated Noetherian scheme.
  If \(X\) is quasi-excellent and of equal characteristic zero,
  then the natural morphisms
  \[
    \mathcal{R}\rho_{X,\mathrm{eh}*}\cO_{X_\mathrm{eh}} \longrightarrow
    \mathcal{R}\rho_{X,\mathrm{sdh}*}\cO_{X_\mathrm{sdh}} \longrightarrow
    \mathcal{R}\rho_{X,\mathrm{h}*}\cO_{X_\mathrm{h}}
  \]
  are quasi-isomorphisms.
  The composition is a quasi-isomorphism by the proof of \cite[Proposition
  6.1]{HJ14} (\cite{HJ14} uses resolutions of singularities, which exist by
  \cite[Theorem 1.1]{Tem08}).
  The second morphism is a quasi-isomorphism since the sdh and h topologies
  coincide in equal characteristic zero \cite[Remark 6.4]{HKK17}.
\end{remark}
\par We now define Du Bois singularities with respect to Grothendieck topologies.
When \(\Sigma = \emptyset\), this definition has appeared before.
The definition using the h topology is used to characterize Du Bois
singularities for complex varieties in \cite[Theorem 4.16]{Lee09}.
In arbitrary characteristic,
the definition using the cdh topology is due to Huber and
Kelly \cite[Definition 7.19]{Hub16},
and the definition using the eh topology appears in \cite[Definition 2.13]{KW}.
\begin{definition}\label{def:dbviatop}
  Let \((X,\Sigma)\) be a pair and let \(\tau \in
  \{\mathrm{rh},\mathrm{cdh},\mathrm{eh},\mathrm{sdh},\mathrm{h}\}\).
  We say that \textsl{the pair \((X,\Sigma)\) has Du Bois singularities with
  respect to the \(\tau\) topology} if the morphism
  \[
    \mathcal{I}_{\Sigma \subseteq X} \longrightarrow \underline{\Omega}^0_{X,\Sigma,\tau}
  \]
  defined in \eqref{eq:omegafunctorial} is a quasi-isomorphism.
  If \((X,\emptyset)\) is a Du Bois pair with respect to the \(\tau\) topology,
  we say that \textsl{\(X\) has Du Bois singularities  with respect to the
  \(\tau\) topology}.
  The \textsl{Du Bois defect of the pair \((X,\Sigma)\) with respect to the
  \(\tau\) topology} is the mapping cone
  \[
    \underline{\Omega}^\times_{X,\Sigma,\tau}
    \coloneqq \operatorname{Cone}\bigl(\mathcal{I}_{\Sigma \subseteq X}
    \longrightarrow \underline{\Omega}^0_{X,\Sigma,\tau}\bigr).
  \]
\end{definition}
For quasi-excellent separated Noetherian schemes of equal characteristic zero, this
definition matches the usual definition of Du Bois singularities and pairs
using the 0-th graded piece of the Deligne--Du Bois complex.
The statement for the h topology is essentially due to Lee
\cite[Theorem 4.16]{Lee09} and 
Huber and J\"order \cite[Corollary 6.16 and Theorem 7.12]{HJ14}.
The statement for the cdh topology is essentially due to Corti\~{n}as,
Haesemeyer, Walker, and Weibel \cite[Lemma 2.1]{CHWW11}.
\begin{proposition}\label{prop:sameashyperresdef}
  Let \(X\) be a quasi-excellent separated Noetherian scheme of equal
  characteristic zero.
  Let \(\pi_\bullet\colon X_\bullet \to X_\mathrm{red}\) be a simplicial or
  (possibly iterated) cubical hyperresolution.
  Let \(\tau \in \{\mathrm{cdh},\mathrm{eh},\mathrm{sdh},\mathrm{h}\}\).
  Then, we have
  \[
    \mathcal{R}\pi_{\bullet*}\cO_{X_\bullet} \cong
    \mathcal{R}\rho_{X,\tau*}\cO_{X_\tau}.
  \]
  As a consequence, we have the following.
  \begin{enumerate}[label=\((\roman*)\),ref=\roman*]
    \item\label{prop:sameashyperresdefeft}
      Suppose that \((X,\Sigma)\) is a pair where \(X\) is essentially of
      finite type over a field of equal characteristic zero.
      Then, \((X,\Sigma)\) has Du Bois singularities in the sense of
      \emph{\citeleft\citen{Kov11}\citemid Definition 3.13\citepunct
      \citen{KS16}\citemid Definition 2.8\citeright}
      if and only if \((X,\Sigma)\) has Du Bois
      singularities with respect to the \(\tau\) topology.
    \item\label{prop:sameashyperresdefqe}
      Suppose that \(\Sigma = \emptyset\).
      Then, \(X\) has Du Bois singularities in the sense of
      \emph{\citeleft\citen{Ste83}\citemid
      (3.5)\citepunct \citen{Murinj}\citemid Definition 7.3.1\citeright}
      if and only if \(X\) has Du Bois singularities
      with respect to the \(\tau\) topology.
  \end{enumerate}
\end{proposition}
\begin{proof}
  By \cite[Corollary 7.2.8]{Murinj}, the left-hand side is the same regardless
  of whether \(X_\bullet\) is a simplicial or a (possibly iterated) cubical
  hyperresolution.
  By Remark \ref{rem:naturalisosehcdhrh}, it suffices to show the cases when
  \(\tau \in \{\mathrm{cdh},\mathrm{h}\}\).
  \par We claim we may
  replace \(X\) by the spectrum of one of its local rings to assume that
  \(X\) is the spectrum of a quasi-excellent local \(\mathbb{Q}\)-algebra, which
  in particular is of finite Krull dimension.
  Since the left-hand side is local, it suffices to show that the cdh and h
  topologies are compatible with localization.
  This follows from
  Nayak's version of Nagata compactification for separated essentially of finite
  type morphisms \cite[Theorem 4.1]{Nay09}, which implies that
  standard cdh covers from \cite[Proposition 5.9]{SV00} and the standard h
  covers from \citeleft\citen{GL01}\citemid Corollary 3.9\citepunct
  \citen{stacks-project}\citemid
  \href{https://stacks.math.columbia.edu/tag/0DBD}{Tag 0DBD}\citeright\ 
  defined over a localization of \(X\) can be extended to standard ph or
  h covers of \(X\).
  \par For the case when \(\tau = \mathrm{h}\), the proof of \cite[Corollary
  6.16]{HJ14} applies.
  While the proof is stated for varieties over a field of characteristic zero,
  the results from \cite{Gei06} used in \cite{HJ14} hold in our situation:
  The required strong form of resolutions of singularities \cite[Definition
  2.4]{Gei06} holds by 
  \cite[Chapter I, \S3, Main Theorem I\((n)\)]{Hir64}, and the cited
  result \cite[Chapitre IV, Th\'eor\`eme 1.2.1]{Gro85} (see also
  \cite[Proposition 3.3]{GNA02}) holds for the structure sheaves \(\cO\).
  \par For the case when 
  \(\tau = \mathrm{cdh}\), the proof of \cite[Lemma 2.1]{CHWW11} applies
  because simplicial hyperresolutions exist in our situation
  \cite[Corollary 4.5.5]{Murinj}.
  The result \cite[Corollary 2.5]{CHW08} cited in \cite{CHWW11} can be replaced
  by the earlier result \cite[Proposition 6.3]{CHSW08} to avoid discussing
  differential forms on \(X\).
  \par Finally, statements \((\ref{prop:sameashyperresdefeft})\) and
  \((\ref{prop:sameashyperresdefqe})\) follow by comparing the definitions
  in Definition \ref{def:dbviatop} and in \citeleft\citen{Ste83}\citemid
  (3.5)\citepunct \citen{Kov11}\citemid Definition 3.13\citepunct
  \citen{KS16}\citemid Definition 2.8\citepunct
  \citen{Murinj}\citemid Definition 7.3.1\citeright.
\end{proof}

\subsection{\texorpdfstring{\(\tau\)}{\unichar{"03C4}}-injective pairs}
In the next section, we will characterize Du Bois pairs in equal characteristic
zero in terms of injectivity of maps on local cohomology.
This definition is inspired by the
proofs of \cite[Theorem 5.4]{Kov11} and \cite[Theorem 2.5]{Kov12},
and this injectivity condition is used to characterize Du Bois singularities
in \citeleft\citen{Kov99}\citemid Lemma 2.2\citepunct
\citen{BST17}\citemid Theorem 4.8\citeright\ when \(X\) is essentially of finite
type over a field and \(\Sigma = \emptyset\).
\par When \(X\) is of prime characteristic \(p >
0\) and \(\Sigma = \emptyset\), this injectivity condition for \(\tau =
\mathrm{h}\) characterizes
\(F\)-injectivity by \cite[Theorem 4.8]{BST17} and its proof (see
\citeleft\citen{BST17}\citemid Theorem 3.3\citepunct
\citen{BS17}\citemid Theorem 4.1\((i)\)\citepunct
\citen{stacks-project}\citemid
\href{https://stacks.math.columbia.edu/tag/0EVW}{Tag
0EVW}\citeright).
A similar definition for \(\tau = \mathrm{eh}\) and \(\Sigma = \emptyset\) is
called \textsl{pseudo-Du Bois} in \cite[Definition 3.4]{KW}.
\begin{definition}
  Let \((X,\Sigma)\) be a pair and let \(\tau \in
  \{\mathrm{rh},\mathrm{cdh},\mathrm{eh},\mathrm{sdh},\mathrm{h}\}\).
  We say that \textsl{the pair \((X,\Sigma)\) is \(\tau\)-injective} at a point
  \(x \in X\) if the natural morphism
  \[
    H^i_x\bigl(\Spec(\cO_{X,x}),\mathcal{I}_{\Sigma,x}\bigr)
    \longrightarrow \mathbb{H}^i_x\bigl(\Spec(\cO_{X,x}),
    \underline{\Omega}^0_{X,\Sigma,\tau,x}\bigr)
  \]
  induced by \eqref{eq:omegafunctorial} is injective.
  We say that \textsl{the pair \((X,\Sigma)\) is \(\tau\)-injective} if it is
  \(\tau\)-injective at every point \(x \in X\).
\end{definition}
\begin{remark}
  Let \(\tau,\tau' \in
  \{\mathrm{rh},\mathrm{cdh},\mathrm{eh},\mathrm{sdh},\mathrm{h}\}\) be a pair of
  topologies such that \(\tau'\) is finer than \(\tau\), or in other words,
  there is an arrow \(\tau \to \tau'\) in the diagram from Definition
  \ref{def:grotops}.
  Then, we have the factorization
  \[
    \begin{tikzcd}[row sep=0]
      & \mathbb{H}^i_x\bigl(\Spec(\cO_{X,x}),
      \underline{\Omega}^0_{X,\Sigma,\tau,x}\bigr) \arrow{dd}\\
      H^i_x\bigl(\Spec(\cO_{X,x}),\mathcal{I}_{\Sigma,x}\bigr)
      \arrow[start anchor={[yshift=-4pt]north east},end anchor={[yshift=-2pt]west}]{ur}
      \arrow[start anchor={[yshift=4pt]south east},end anchor={[yshift=2pt]west}]{dr}\\
      & \mathbb{H}^i_x\bigl(\Spec(\cO_{X,x}),
      \underline{\Omega}^0_{X,\Sigma,\tau',x}\bigr)\mathrlap{.}
    \end{tikzcd}
  \]
  Thus, \(\tau'\)-injectivity implies \(\tau\)-injectivity.
  If \(X\) is quasi-excellent of equal characterisic zero, then
  cdh-, eh-, sdh-, and h-injectivity all coincide by Remark
  \ref{rem:naturalisosehcdhrh} and Proposition \ref{prop:sameashyperresdef}.
\end{remark}

\section{The key injectivity theorem}
\label{sect:dbinchar0}
In this section, we prove our version (Theorem \ref{thm:keyinjintro})
of the key injectivity theorem of Kov\'acs
and Schwede
\citeleft\citen{KS16deforms}\citemid Theorem 3.3\citepunct \citen{KS16}\citemid 
Theorem 3.2\citepunct \citen{MSS17}\citemid Lemma 3.2\citepunct
\citen{KK20}\citemid Theorem 6.3 and Corollary 6.5\citeright.
As in \cite{MurKV,Murinj}, the idea is that we want to use Noetherian
approximation to reduce Theorem \ref{thm:keyinjintro} to the case when \(X\) is
of finite type over a field.\smallskip
\par A new idea in this paper is that we will use Zariski--Riemann spaces
associated to pairs \cite{FK06,Tem10,FK18} (see Definition \ref{def:zrforpairs})
to construct an analogue of the Nagata compactification of a separated finite
type morphism of schemes.
For a quasi-compact quasi-separated scheme \(X\) over a Noetherian ring
\(\mathbf{k}\), our construction results in a locally ringed space \(\langle X
\rangle_{\mathrm{cpt}}\) that is a inverse limit of proper
\(\mathbf{k}\)-schemes such that \(\langle X
\rangle_{\mathrm{cpt}} \setminus X\) is ind-constructible (in fact, a union of
closed subsets).
We can then prove a version (Theorem \ref{thm:limitofhodge})
of the Hodge-theoretic input for \(\langle X
\rangle_{\mathrm{cpt}}\) that is used in the proofs in
\cite{KS16deforms,KS16,MSS17} using Grothendieck's limit theorem for inverse
limits of toposes \cite[Expos\'e VI]{SGA42}.
For varieties, this Hodge-theoretic statement is a consequence of the \(E_1\)
degeneration of the Hodge-to-de Rham spectral sequence \cite{Del74,DB81} (see
Theorem \ref{thm:appofhodgetodr}).
We then use the long exact sequence on local cohomology to deduce
Theorem \ref{thm:keyinjintro} from Theorem \ref{thm:limitofhodge}.\smallskip
\par Finally, we use our key injectivity theorem to prove that in equal
characteristic zero, having Du Bois singularities with respect to the h topology
is the same thing as being h-injective (Theorem \ref{thm:splittingcriterion}).
For quasi-excellent schemes of equal characteristic zero, the same statement
holds for the cdh, eh, and sdh topologies.

\subsection{Zariski--Riemann spaces}
We define Zariski--Riemann spaces associated to pairs.
These are called \textsl{relative Riemann--Zariski spaces} in \cite{Tem10}.
\begin{citeddef}[{\citeleft\citen{FK06}\citemid Definition 5.3 and
  Definition 5.9\citepunct
  \citen{Tem10}\citemid \S3.3\citepunct
  \citen{FK18}\citemid Chapter II, Definition E.2.2\citeright}]
  \label{def:zrforpairs}
  Let \(X\) be a quasi-compact quasi-separated scheme.
  Let \(\mathcal{I} \subseteq \mathcal{O}_X\) be a quasi-coherent ideal sheaf of
  finite type such that \(U = X - V(\mathcal{I})\) is a dense open subset of
  \(S\).
  Denote by \(\operatorname{AId}_{(X,U)}\) the set of quasi-coherent ideal
  sheaves \(\mathcal{J} \subseteq \mathcal{O}_X\) of finite type such that
  \(\lvert V(\mathcal{J}) \rvert \subseteq \lvert V(\mathcal{I}) \rvert\).
  The \textsl{Zariski--Riemann space associated to the pair \((X,U)\)} is the
  inverse limit
  \[
    \langle X \rangle_U \coloneqq \varprojlim_{\mathcal{J} \in
    \operatorname{AId}_{(X,U)}} \operatorname{Bl}_{\mathcal{J}}X
  \]
  computed in the category of locally ringed spaces.
  Because the blowup morphisms are isomorphisms along \(U\), the
  projection morphism \(\langle X \rangle_U \to X\) induces an isomorphism
  between an open subspace of \(\langle X \rangle_U\) and \(U\).
\end{citeddef}
We can use Zariski--Riemann spaces to define canonical compactifications.
\begin{citeddef}[{\cite[Definition F.2.1]{FK18}}]
  Let \(f\colon X \to Y\) be a separated finite type morphism of quasi-compact
  quasi-separated schemes.
  By Nagata's compactification theorem \cite{Nag63,Con07},
  there exists a commutative diagram
  \[
    \begin{tikzcd}[column sep=small]
      X \arrow[hook]{rr}{j}\arrow{dr}[swap,pos=0.3]{f\vphantom{\bar{f}}}
      & & \bar{X} \arrow{dl}[pos=0.3]{\bar{f}}\\
      & Y
    \end{tikzcd}
  \]
  where \(j\) is a dense open immersion and \(\bar{f}\) is proper.
  The \textsl{canonical compactification of \(X\) over \(Y\)} is
  \[
    \langle X \rangle_{\mathrm{cpt}} \coloneqq \langle \bar{X} \rangle_X.
  \]
\end{citeddef}
\begin{remark}
  We have defined \(\langle X \rangle_{\mathrm{cpt}}\) using Nagata's
  compactification theorem as input.
  In \cite[Chapter II,
  Definition F.2.9]{FK18}, Fujiwara and Kato define the canonical
  compactification \(\langle X \rangle_{\mathrm{cpt}}\) without assuming
  Nagata's compactification theorem and use the notion to give
  an alternative proof of Nagata's compactification theorem \cite[Chapter II,
  Theorem F.1.1]{FK18}.
\end{remark}
\subsection{The Hodge-theoretic input}
Our goal in this section is to prove a version of the surjectivity in
\eqref{eq:hodgesurj} below for separated
schemes that are not proper or even of finite
type over a field.
\begin{citedthm}[{\citeleft\citen{Del74}\citemid \S8.2\citepunct
  \citen{DB81}\citemid Th\'eor\`eme
  4.5\((iii)\) and Proof of Th\'eor\`eme 4.6\citepunct
  \citen{Kov11}\citemid Theorem 4.1 and Corollary 4.2\citeright}]
  \label{thm:appofhodgetodr}
  Let \((X,\Sigma)\) be a pair such that \(X\) is proper over \(\mathbb{C}\).
  Consider the Deligne--Du Bois complex
  \((\underline{\Omega}_{X,\Sigma}^\bullet,F)\)
  and set \(U \coloneqq X - \Sigma\).
  Then, the spectral sequence
  \[
    E_1^{p,q} =
    \mathbb{H}^q\bigl(X,\operatorname{Gr}^p_F
    \underline{\Omega}_{X,\Sigma}^\bullet
    \bigr)
    = \mathbb{H}^q\bigl(X,\underline{\Omega}_{X,\Sigma}^p\bigr) \Rightarrow
    H^{p+q}_c(U^{\mathrm{an}},\mathbb{C})
  \]
  of the filtration \(F\) degenerates on the \(E_1\) page and abuts to the Hodge
  filtration on the mixed Hodge structure
  \(H^{p+q}_c(U^{\mathrm{an}},\mathbb{C})\).
  As a consequence, the natural morphisms
  \begin{equation}\label{eq:hodgesurj}
    H^i(X,\mathcal{I}_\Sigma) \longrightarrow
    \mathbb{H}^i\bigl(X,\underline{\Omega}_{X,\Sigma}^0\bigr)
  \end{equation}
  are surjective for every \(i\).
\end{citedthm}
To state our version of the surjectivity \eqref{eq:hodgesurj}, we define our
version of the canonical compactification \(\langle X \rangle_{\mathrm{cpt}}\)
for quasi-compact separated schemes over Noetherian rings.
\begin{setup}\label{setup:invlimcpt}
  Let \(X\) be a quasi-compact separated scheme over a Noetherian ring
  \(\mathbf{k}\).
  Consider a closed subscheme \(Z \subseteq X\) such that \(\mathcal{I}_Z\) is
  of finite type.
  By absolute Noetherian approximation \cite[Theorem C.9]{TT90}, we can write
  \[
    X \cong \varprojlim_{\lambda \in \Lambda} X_\lambda
  \]
  as an inverse limit of separated schemes of finite type over \(\mathbf{k}\)
  with affine, scheme-theoretically dominant transition morphisms
  \(u_{\lambda\mu}\colon X_\mu \to X_\lambda\) and with
  projection morphisms \(u_\lambda\colon X \to X_\lambda\).
  \par By \cite[Proposition 8.6.3]{EGAIV3}, there exists an index \(\lambda_0
  \in \Lambda\) such that \(\Sigma =
  u_{\lambda_0}^{-1}(\Sigma_{\lambda_0})\) for a closed subscheme
  \(\Sigma_{\lambda_0} \subseteq X_{\lambda_0}\).
  Choose a Nagata compactification \(X_{\lambda_0} \hookrightarrow
  \bar{X}_{\lambda_0}\) over \(\mathbf{k}\)
  and let \(\bar{\Sigma}_{\lambda_0}\) be the closure
  of \(\Sigma_{\lambda_0}\) in \(\bar{X}_{\lambda_0}\).
  Denote by
  \[
    \pi_{\lambda_0}\colon \langle X_{\lambda_0} \rangle_{\operatorname{cpt}}
    \xrightarrow{v_{\lambda_0,\mathcal{J}}}
    \operatorname{Bl}_{\mathcal{J}} \bar{X}_{\lambda_0}
    \xrightarrow{\pi_{\lambda_0,\mathcal{J}}}
    \bar{X}_{\lambda_0}
  \]
  the projection morphisms for \(\mathcal{J} \in
  \operatorname{AId}_{(\bar{X}_{\lambda_0},X_{\lambda_0})}\),
  and let \(\pi_{\lambda_0}\) be their composition.
  We then set
  \[
    \mathcal{I}_{\langle \Sigma_{\lambda_0} \rangle} \coloneqq
    \pi_{\lambda_0}^{-1}
    \mathcal{I}_{\bar{\Sigma}_{\lambda_0}} \cdot
    \cO_{\langle X_{\lambda_0} \rangle_{\mathrm{cpt}}}
    = v_{\lambda_0,\mathcal{J}}^{-1}\pi_{\lambda_0,\mathcal{J}}^{-1}
    \mathcal{I}_{\bar{\Sigma}_{\lambda_0}} \cdot
    \cO_{\langle X_{\lambda_0} \rangle_{\mathrm{cpt}}}.
  \]
  For each \(\lambda \ge \lambda_0\), choose a Nagata compactification
  \(X_\lambda \hookrightarrow \bar{X}_\lambda\) over \(\mathbf{k}\)
  that fits into the commutative diagram
  \[
    \begin{tikzcd}[column sep=3em]
      X_\lambda \rar[hook]\dar[swap]{u_{\lambda_0\lambda}}
      & \bar{X}_\lambda \dar[swap]{\bar{u}_{\lambda_0\lambda}}
      & \lar[swap]{\pi_\lambda} \langle X_\lambda \rangle_{\mathrm{cpt}}
      \dar{\langle u_{\lambda_0\lambda} \rangle}\\
      X_{\lambda_0} \rar[hook] & \bar{X}_{\lambda_0}
      & \lar[swap]{\pi_{\lambda_0}} \langle X_{\lambda_0} \rangle_{\mathrm{cpt}}
    \end{tikzcd}
  \]
  where we denote
  \[
    \pi_{\lambda}\colon \langle X_{\lambda} \rangle_{\operatorname{cpt}}
    \xrightarrow{v_{\lambda,\mathcal{J}}}
    \operatorname{Bl}_{\mathcal{J}} \bar{X}_{\lambda}
    \xrightarrow{\pi_{\lambda,\mathcal{J}}}
    \bar{X}_{\lambda}
  \]
  as before for \(\mathcal{J} \in
  \operatorname{AId}_{(\bar{X}_{\lambda},X_{\lambda})}\).
  By the universal property of blowups
  \cite[\href{https://stacks.math.columbia.edu/tag/0806}{Tag
  0806}]{stacks-project}, the morphisms \(\langle u_{\lambda_0\lambda} \rangle\)
  of locally ringed spaces exist and fit into the commutative diagram above.
  Using the universal property of blowups
  again, we obtain the inverse system
  \[
    \Set[\Big]{\langle u_{\lambda\mu} \rangle\colon
    \langle X_\mu \rangle_{\mathrm{cpt}}
    \longrightarrow \langle X_\lambda \rangle_{\mathrm{cpt}}}_{\mu \ge
    \lambda \ge \lambda_0}
  \]
  of locally ringed spaces.
  We then set
  \[
    \langle X \rangle_{\mathrm{cpt}} \coloneqq \varprojlim_{\lambda \in \Lambda}
    \langle X_\lambda \rangle_{\mathrm{cpt}}
  \]
  where the inverse limit is computed in the category of locally ringed spaces
  with projection morphisms
  \[
    \langle u_\lambda \rangle
    \colon \langle X \rangle_{\mathrm{cpt}} \longrightarrow
    \langle X_\lambda \rangle_{\mathrm{cpt}}.
  \]
  Choosing a Nagata compactification \(X_\lambda \hookrightarrow
  \bar{X}_\lambda\) for each \(\lambda\), we also set
  \begin{equation}\label{eq:honcptaslim}
    \langle X \rangle_{\mathrm{cpt},\mathrm{h}} \coloneqq \varprojlim_{\lambda
    \in \Lambda} \varprojlim_{\mathcal{J} \in
    \operatorname{AId}_{(\bar{X}_\lambda,X_\lambda)}}
    \bigl(\operatorname{Bl}_{\mathcal{J}}\bar{X}_{\lambda}\bigr)_{\mathrm{h}}
  \end{equation}
  where the inverse limit is computed as sites as in \cite[Expos\'e VI,
  Th\'eor\`eme 8.2.3]{SGA42}.
  This definition does not depend on the choice of compactification
  \(\bar{X}_\lambda\) because for any other choice of
  compactification \(\bar{X}'_\lambda\), the inverse systems of blowups
  \(\operatorname{Bl}_{\mathcal{J}}\bar{X}_{\lambda}\) and
  \(\operatorname{Bl}_{\mathcal{J}}\bar{X}_{\lambda}'\) appearing in the inverse
  limit above are coinitial by \cite[Premi\`ere partie, Corollaire 5.7.12]{RG71}
  (see also \cite[Theorem 2.11]{Con07}).
  \par Finally, for each \(\lambda \ge \lambda_0\), we set
  \begin{gather*}
    \begin{aligned}
      \mathcal{I}_{\bar{\Sigma}_\lambda} &\coloneqq
      \bar{u}_{\lambda_0\lambda}^{-1}\mathcal{I}_{\bar{\Sigma}_{\lambda_0}} \cdot
      \cO_{\bar{X}_\lambda}\\
      \mathcal{I}_{\langle \Sigma_\lambda \rangle} &\coloneqq
      \pi_{\lambda}^{-1}\mathcal{I}_{\bar{\Sigma}_{\lambda}} \cdot
      \cO_{\langle X_\lambda \rangle_{\mathrm{cpt}}}
    \end{aligned}\\
    \mathcal{I}_{\langle \Sigma \rangle} \coloneqq
    \langle u_{\lambda}\rangle^{-1}\mathcal{I}_{\langle \Sigma_{\lambda}
    \rangle} \cdot \cO_{\langle X \rangle_{\mathrm{cpt}}}.
  \end{gather*}
  We also let \(\mathcal{I}_{\bar{\Sigma}_\lambda,\mathrm{h}}\),
  \(\mathcal{I}_{\langle \Sigma_\lambda \rangle,\mathrm{h}}\), and
  \(\mathcal{I}_{\langle \Sigma \rangle,\mathrm{h}}\) be their
  h-sheafifications.
  Note that \(\mathcal{I}_{\bar{\Sigma}_\lambda,\mathrm{h}}\),
  \(\mathcal{I}_{\langle \Sigma_\lambda \rangle,\mathrm{h}}\), and
  \(\mathcal{I}_{\langle \Sigma \rangle,\mathrm{h}}\) are ideal sheaves in
  \(\cO_{\bar{X}_{\lambda,\mathrm{h}}}\),
  \(\cO_{\langle X_\lambda \rangle_{\mathrm{cpt},\mathrm{h}}}\), and
  \(\cO_{\langle X \rangle_{\mathrm{cpt},\mathrm{h}}}\) since
  sheafification is exact \citeleft\citen{SGA41}\citemid
  Expos\'e II, Th\'eor\`eme 4.1(1)\citepunct \citen{stacks-project}\citemid 
  \href{https://stacks.math.columbia.edu/tag/00WJ}{Tag 00WJ}\citeright.
\end{setup}
We can now state and prove our version of the surjectivity \eqref{eq:hodgesurj}
in Theorem \ref{thm:appofhodgetodr}.
\begin{theorem}\label{thm:limitofhodge}
  Let \(X\) be a quasi-compact separated scheme of equal characteristic zero,
  and let
  \(\Sigma \subseteq X\) be a closed subscheme such that \(\mathcal{I}_\Sigma\)
  is of finite type.
  Fix notation as in Setup \ref{setup:invlimcpt} where \(\mathbf{k} =
  \mathbb{Q}\).
  Then, the morphism
  \[
    H^i\bigl(\langle X \rangle_{\mathrm{cpt}},\mathcal{I}_{\langle \Sigma
    \rangle}\bigr)
    \longrightarrow
    H^i\bigl(\langle X \rangle_{\mathrm{cpt},\mathrm{h}},
    \mathcal{I}_{\langle \Sigma \rangle,\mathrm{h}}\bigr)
  \]
  is surjective for every \(i\).
\end{theorem}
\begin{proof}
  We have the commutative diagram
  \begin{equation}\label{eq:htopcohmaplambdas}
    \begin{tikzcd}[row sep=2.25em,baseline=(midrow.base)]
      H^i\Bigl(\operatorname{Bl}_{\mathcal{J}}\bar{X}_\lambda,
      \pi_{\lambda,\mathcal{J}}^{-1}\mathcal{I}_{\bar{\Sigma}_\lambda}\cdot
      \cO_{\operatorname{Bl}_{\mathcal{J}}\bar{X}_\lambda}\Bigr)
      \rar \arrow[d,""{name=midrow}] &
      H^i\biggl( \bigl(\operatorname{Bl}_{\mathcal{J}}\bar{X}_\lambda
      \bigr)_{\mathrm{h}},\Bigl(
      \pi_{\lambda,\mathcal{J}}^{-1}\mathcal{I}_{\bar{\Sigma}_\lambda}\cdot
      \cO_{\operatorname{Bl}_{\mathcal{J}}\bar{X}_\lambda}\Bigr)_{\mathrm{h}}
      \biggr) \dar\\
      H^i\bigl(\langle X
      \rangle_{\mathrm{cpt}},\mathcal{I}_{\langle\Sigma\rangle}\bigr)
      \rar &
      H^i\bigl(\langle X \rangle_{\mathrm{cpt},\mathrm{h}},
      \mathcal{I}_{\langle\Sigma\rangle,\mathrm{h}}\bigr)\mathrlap{.}
    \end{tikzcd}
  \end{equation}
  \par We claim that the top horizontal map in \eqref{eq:htopcohmaplambdas} is
  surjective.
  We first note that
  \[
    H^i\biggl( \bigl(\operatorname{Bl}_{\mathcal{J}}\bar{X}_\lambda
    \bigr)_{\mathrm{h}},\Bigl(
    \pi_{\lambda,\mathcal{J}}^{-1}\mathcal{I}_{\bar{\Sigma}_\lambda}\cdot
    \cO_{\operatorname{Bl}_{\mathcal{J}}\bar{X}_\lambda}\Bigr)_{\mathrm{h}}
    \biggr) \cong \mathbb{H}^i\Bigl(\operatorname{Bl}_{\mathcal{J}}
    \bar{X}_\lambda,\underline{\Omega}^0_{\operatorname{Bl}_{\mathcal{J}}
    \bar{X}_\lambda,\bar{\Sigma}_\lambda}\Bigr)
  \]
  by the definition of \(\underline{\Omega}^0_{X,\Sigma,\mathrm{h}}\),
  Proposition \ref{prop:sameashyperresdef},
  and the definition of \(\underline{\Omega}^0_{X,\Sigma}\).
  Since the formation of \(\underline{\Omega}^0_{X,\Sigma}\) is compatible with
  regular base change \cite[Corollary 7.2.7]{DB81} and by flat base change,
  it suffices to show that the
  top horizontal map is surjective after base change to \(\mathbb{C}\).
  After this reduction,
  we see that the top horizontal map in \eqref{eq:htopcohmaplambdas} is
  surjective by Theorem \ref{thm:appofhodgetodr}.
  \par To finish the proof, our goal is to apply Grothendieck's
  limit theorem \cite[Expos\'e VI, Corollaire 8.7.7 and Remarque 8.7.8]{SGA42}
  to the maps in the top row of \eqref{eq:htopcohmaplambdas} to obtain the
  surjectivity of the bottom horizontal map in \eqref{eq:htopcohmaplambdas}.
  We note that the h-topology over a separated Noetherian scheme
  is coherent in the sense of \cite[Expos\'e VI, D\'efinition 2.3]{SGA42} by
  \citeleft\citen{GL01}\citemid \S3\citepunct
  \citen{GK15}\citemid Remark 2.4\citeright\ and that the morphisms of the
  associated toposes in \eqref{eq:honcptaslim}
  are coherent in the sense of \cite[Expos\'e VI,
  D\'efinition 3.1]{SGA42} by \cite[Expos\'e VI, Corollaire 3.3]{SGA42}.
  Thus, taking cohomology is compatible with the inverse limit in
  \eqref{eq:honcptaslim} by Grothendieck's limit theorem for inverse limits of
  toposes \cite[Expos\'e VI, Corollaire 8.7.7 and Remarque 8.7.8]{SGA42}.
  We now note that
  \begin{align*}
    \mathcal{I}_{\langle\Sigma\rangle}
    &\cong \varinjlim_{\lambda \ge \lambda_0}
    \varinjlim_{\mathcal{J} \in
    \operatorname{AId}_{(\bar{X}_\lambda,X_\lambda)}}
    \langle v_{\lambda,\mathcal{J}} \rangle^{-1}
    \langle \pi_{\lambda,\mathcal{J}} \rangle^{-1}
    \mathcal{I}_{\bar{\Sigma}_{\lambda}} \cdot \cO_{\langle X
    \rangle_{\mathrm{cpt}}}
    \intertext{which also implies}
    \mathcal{I}_{\langle\Sigma\rangle,\mathrm{h}}
    &\cong \varinjlim_{\lambda \ge \lambda_0}
    \varinjlim_{\mathcal{J} \in
    \operatorname{AId}_{(\bar{X}_\lambda,X_\lambda)}}
    \Bigl(
    \langle v_{\lambda,\mathcal{J}} \rangle^{-1}
    \langle \pi_{\lambda,\mathcal{J}} \rangle^{-1}
    \mathcal{I}_{\bar{\Sigma}_{\lambda}} \cdot \cO_{\langle X
    \rangle_{\mathrm{cpt}}}
    \Bigr)_{\mathrm{h}}
  \end{align*}
  by the fact that
  h-sheafification commutes with direct limits
  \cite[\href{https://stacks.math.columbia.edu/tag/00WI}{Tag
  00WI}]{stacks-project}.
  Thus, Grothendieck's limit theorem says that the morphisms
  \begin{align*}
    \varinjlim_{\lambda \ge \lambda_0}
    \varinjlim_{\mathcal{J} \in
    \operatorname{AId}_{(\bar{X}_\lambda,X_\lambda)}}
    H^i\Bigl(\operatorname{Bl}_{\mathcal{J}}\bar{X}_\lambda,
    \pi_{\lambda,\mathcal{J}}^{-1}\mathcal{I}_{\bar{\Sigma}_\lambda}\cdot
    \cO_{\operatorname{Bl}_{\mathcal{J}}\bar{X}_\lambda}\Bigr)
    &\overset{\sim}{\longrightarrow}
    H^i\bigl(\langle X
    \rangle_{\mathrm{cpt}},\mathcal{I}_{\langle\Sigma\rangle}\bigr)\\
    \varinjlim_{\lambda \ge \lambda_0}
    \varinjlim_{\mathcal{J} \in
    \operatorname{AId}_{(\bar{X}_\lambda,X_\lambda)}}
H^i\biggl( \bigl(\operatorname{Bl}_{\mathcal{J}}\bar{X}_\lambda
      \bigr)_{\mathrm{h}},\Bigl(
      \pi_{\lambda,\mathcal{J}}^{-1}\mathcal{I}_{\bar{\Sigma}_\lambda}\cdot
      \cO_{\operatorname{Bl}_{\mathcal{J}}\bar{X}_\lambda}\Bigr)_{\mathrm{h}}
      \biggr)
    &\overset{\sim}{\longrightarrow}
    H^i\bigl(\langle X \rangle_{\mathrm{cpt},\mathrm{h}},
    \mathcal{I}_{\langle\Sigma\rangle,\mathrm{h}}\bigr)
  \end{align*}
  coming from taking direct limits of the vertical maps in
  \eqref{eq:htopcohmaplambdas} are isomorphisms.
  Combined with the surjectivity proved in the previous paragraph,
  this shows that the bottom horizonal map in \eqref{eq:htopcohmaplambdas} is
  surjective.
\end{proof}
\subsection{The key injectivity theorem}
We now prove our version of the key injectivity theorem of Kov\'acs and
Schwede.
When \(X\) is essentially of finite type over a field of characteristic zero,
a stronger statement which removes the
assumption that \(X \setminus \{x\}\) is Du Bois is known.
This stronger injectivity theorem is proved in
\cite[Theorem 3.3]{KS16deforms} when \(\Sigma =
\emptyset\), in \cite[Theorem 3.2]{KS16} for reduced pairs, and in \cite[Lemma
3.2]{MSS17} in general.
See also \cite[Theorem 6.3 and Corollary 6.5]{KK20}.
Theorem \ref{thm:keyinjintro} is the special case of statement \((\ref{thm:keyinjhtop})\)
when \(\Sigma = \emptyset\).
\begin{theorem}\label{thm:keyinj}
  Let \(X\) be a separated
  Noetherian scheme of equal characteristic zero and let \(\Sigma
  \subseteq X\) be a closed subscheme.
  Let \(\tau \in \{\mathrm{cdh},\mathrm{eh},\mathrm{sdh},\mathrm{h}\}\).
  For every \(x \in X\), consider the natural morphism
  \begin{equation}\label{eq:maponlc}
    H^i_x\bigl(\Spec(\cO_{X,x}),\mathcal{I}_{\Sigma,x}\bigr)
    \longrightarrow \mathbb{H}^i_x\bigl(\Spec(\cO_{X,x}),
    \underline{\Omega}^0_{X,\Sigma,\tau,x}\bigr)
  \end{equation}
  induced by \eqref{eq:omegafunctorial}.
  Suppose one of the following conditions holds.
  \begin{enumerate}[label=\((\roman*)\),ref=\roman*]
    \item\label{thm:keyinjhtop} \(\tau = \mathrm{h}\).
    \item\label{thm:keyinjqe} \(X\) is quasi-excellent.
  \end{enumerate}
  Suppose that \(\Spec(\cO_{X,x}) \setminus \{x\}\) is Du Bois
  with respect to the \(\tau\) topology.
  Then, \eqref{eq:maponlc} is surjective for every \(x \in X\) and every \(i\).
\end{theorem}
\begin{proof}
  By Remark \ref{rem:naturalisosehcdhrh} and Proposition
  \ref{prop:sameashyperresdef}, it suffices to show \((\ref{thm:keyinjhtop})\).
  Since the question is local, we may assume that \(X = \Spec(R)\) is the
  spectrum of a Noetherian local \(\mathbb{Q}\)-algebra \((R,\fm)\) and that \(x
  = \fm\).
  After possibly replacing \(\lambda_0\) by a larger index in \(\Lambda\), we
  may assume that \(\fm = u_{\lambda_0}^{-1}(Z_{\lambda_0})\) for a closed subset
  \(Z_{\lambda_0} \subseteq X_{\lambda_0}\).
  Set \(Z_\lambda \coloneqq u_{\lambda_0\lambda}^{-1}(Z_{\lambda_0})\) for every
  \(\lambda \ge \lambda_0\).
  \par With notation as in Setup \ref{setup:invlimcpt},
  let
  \(\Phi_\lambda \coloneqq \langle X_\lambda \rangle_{\mathrm{cpt}} \setminus
  X_\lambda\) and \(\Phi \coloneqq \langle X \rangle_{\mathrm{cpt}} \setminus X\).
  We then have the commutative diagram
  \begin{equation}\label{eq:keyinjladder}
    \begin{tikzcd}[baseline=(midrow.base)]
      H^{i-1}_{\langle X \rangle_{\mathrm{cpt}} / (\Phi \cup \{\fm\})}
      \bigl(\langle X \rangle_{\mathrm{cpt}},
      \mathcal{I}_{\langle \Sigma \rangle}\bigr) \dar \rar
      &
      \mathbb{H}^{i-1}_{\langle X \rangle_{\mathrm{cpt}} / (\Phi \cup \{\fm\})}
      \bigl(\langle X \rangle_{\mathrm{cpt}},
      \mathcal{R}\rho_{\langle X \rangle_{\mathrm{cpt}},\mathrm{h}*}
      \mathcal{I}_{\langle \Sigma \rangle,\mathrm{h}}\bigr) \dar\\
      H^i_{\Phi \cup \{\fm\}}\bigl(\langle X \rangle_{\mathrm{cpt}},
      \mathcal{I}_{\langle \Sigma \rangle}\bigr) \dar \rar
      &
      \mathbb{H}^i_{\Phi \cup \{\fm\}}\bigl(\langle X \rangle_{\mathrm{cpt}},
      \mathcal{R}\rho_{\langle X \rangle_{\mathrm{cpt}},\mathrm{h}*}
      \mathcal{I}_{\langle \Sigma \rangle,\mathrm{h}}\bigr)
      \arrow[d,""{name=midrow}]\\
      H^i\bigl(\langle X \rangle_{\mathrm{cpt}},\mathcal{I}_{\langle \Sigma
      \rangle}\bigr) \dar \rar[twoheadrightarrow]
      & \mathbb{H}^i\bigl(\langle X \rangle_{\mathrm{cpt}},
      \mathcal{R}\rho_{\langle X \rangle_{\mathrm{cpt}},\mathrm{h}*}
      \mathcal{I}_{\langle \Sigma \rangle,\mathrm{h}}\bigr) \dar\\
      H^{i}_{\langle X \rangle_{\mathrm{cpt}} / (\Phi \cup \{\fm\})}\bigl(\langle X \rangle_{\mathrm{cpt}},
      \mathcal{I}_{\langle \Sigma \rangle}\bigr) \rar
      & \mathbb{H}^{i}_{\langle X \rangle_{\mathrm{cpt}} / (\Phi \cup \{\fm\})}\bigl(\langle X \rangle_{\mathrm{cpt}},
      \mathcal{R}\rho_{\langle X \rangle_{\mathrm{cpt}},\mathrm{h}*}
      \mathcal{I}_{\langle \Sigma \rangle,\mathrm{h}}\bigr)
    \end{tikzcd}
  \end{equation}
  with exact columns, where \(\rho_{\langle X
  \rangle_{\mathrm{cpt}},\mathrm{h}}\colon \langle X
  \rangle_{\mathrm{cpt},\mathrm{h}} \to \langle X \rangle_{\mathrm{cpt}}\)
  is the projection morphism from the inverse limit of h sites to
  the inverse limit of Zariski sites.
  The third horizontal map is surjective by Theorem \ref{thm:limitofhodge}.
  \par We now
  show that the top and bottom horizontal maps in \eqref{eq:keyinjladder}
  are isomorphisms.
  We claim we have the chain of isomorphisms
  \begin{align*}
    H^{i}_{\langle X \rangle_{\mathrm{cpt}} / (\Phi \cup
    \{\fm\})}\bigl(\langle X \rangle_{\mathrm{cpt}},\mathcal{I}_{\langle \Sigma
    \rangle}\bigr)
    &\cong
    \varinjlim_{\lambda \ge \lambda_0}
    H^{i}\Bigl(\langle X \rangle_{\mathrm{cpt}} \setminus \bigl(
        \langle u_{\lambda}
    \rangle^{-1}(\Phi_\lambda) \cup \{\fm\}\bigr),\mathcal{I}_{\langle
    \Sigma \rangle}\Bigr)\\
    &\cong
    \varinjlim_{\lambda \ge \lambda_0}
    \varinjlim_{\mu \ge \lambda}
    H^{i}\Bigl(\langle X_\mu \rangle_{\mathrm{cpt}} \setminus \bigl(
        \langle u_{\lambda\mu}
      \rangle^{-1}(\Phi_\lambda \cup \overline{Z_\lambda})\bigr),\mathcal{I}_{\langle
    \Sigma_\mu \rangle}\Bigr)\\
    &\cong
    \varinjlim_{\mu \ge \lambda_0}
    H^{i}\Bigl(\langle X_\mu \rangle_{\mathrm{cpt}} \setminus (
      \Phi_\mu \cup \overline{Z_\mu}),\mathcal{I}_{\langle
    \Sigma_\mu \rangle}\Bigr)\\
    &=
    \varinjlim_{\mu \ge \lambda_0}
    H^{i}\bigl(X_\mu \setminus Z_\mu,\mathcal{I}_{\Sigma_\mu}\bigr)\\
    &\cong
    H^{i}\bigl(X \setminus \{\fm\},\mathcal{I}_{\Sigma}\bigr).
  \end{align*}
  The first isomorphism holds by \cite[Motif D on p.\ 221]{Har66}, and also
  holds for the complexes \(\mathcal{R}\rho_{\langle X
  \rangle_{\mathrm{cpt}},\mathrm{h}*} \mathcal{I}_{\langle \Sigma
  \rangle,\mathrm{h}}\).
  The second isomorphism holds by Grothendieck's limit theorem
  \cite[Expos\'e VI, Corollaire 8.7.7 and Remarque 8.7.8]{SGA42} where
  the transition morphisms are those
  induced by the morphisms
  \[
    \langle X_\mu \rangle_{\mathrm{cpt}} \setminus\Phi_\mu \cup \overline{Z_\mu} 
    \hooklongrightarrow
    \langle X_\mu \rangle_{\mathrm{cpt}} \setminus \langle u_{\lambda\mu}
    \rangle^{-1}(\Phi_\lambda \cup \overline{Z_\lambda})
    \longrightarrow
    \langle X_\lambda \rangle_{\mathrm{cpt}} \setminus
    (\Phi_\lambda \cup \overline{Z_\lambda})
  \]
  of locally ringed spaces.
  The third isomorphism holds
  since the two direct systems are cofinal.
  The fourth equality holds since
  \(X_\mu = \langle X_\mu \rangle_{\mathrm{cpt}} \setminus \Phi_\mu\).
  The last isomorphism holds by Grothendieck's limit theorem
  \cite[Expos\'e VI, Corollaire 8.7.7 and Remarque 8.7.8]{SGA42}.
  The analogoous chain holds for the correspoding sheaves on the h topology.
  In both topologies, we use the fact that the ideal sheaves on \(\langle X
  \rangle_{\mathrm{cpt}}\) and \(\langle X
  \rangle_{\mathrm{cpt},\mathrm{h}}\) are direct limits of the corresponding
  ideal sheaves on \(\langle X_\lambda \rangle_{\mathrm{cpt}}\) and \(\langle
  X_\lambda \rangle_{\mathrm{cpt},\mathrm{h}}\) (see the last paragraph of the
  proof of Theorem \ref{thm:limitofhodge}).
  Thus, the top and  bottom horizontal maps in \eqref{eq:keyinjladder} are
  isomorphisms by the
  assumption that \((X,\Sigma)\) is Du Bois away from \(\fm\).
  \par By the four lemma, we see that the second horizontal map in
  \eqref{eq:keyinjladder} is surjective.
  Our next step is to realize the map \eqref{eq:maponlc} as a direct
  summand of the second horizontal map in \eqref{eq:keyinjladder}.
  By \cite[Motif D on p.\ 219]{Har66},
  the Mayer--Vietoris sequence \cite[Chapter II, Proof of Proposition
  4.2]{Har75}, and the fact that \(\fm \notin Z\),
  the second horizontal map can be identified with the
  direct sum
  \[
    \begin{tikzcd}[row sep=1em,column sep=1.475em]
      H^i_\Phi\bigl(\langle X \rangle_{\mathrm{cpt}},
      \mathcal{I}_{\langle \Sigma \rangle}\bigr)
      \rar\dar[phantom]{\oplus}
      & \mathbb{H}^i_\Phi\bigl(\langle X \rangle_{\mathrm{cpt}},
      \mathcal{R}\rho_{\langle X \rangle_{\mathrm{cpt}},\mathrm{h}*}
      \mathcal{I}_{\langle \Sigma \rangle,\mathrm{h}}\bigr)
      \dar[phantom]{\oplus}\\
      H^i_{\fm}\bigl(\langle X \rangle_{\mathrm{cpt}},
      \mathcal{I}_{\langle \Sigma \rangle}\bigr) \rar
      & \mathbb{H}^i_{\fm}\bigl(\langle X \rangle_{\mathrm{cpt}},
      \mathcal{R}\rho_{\langle X \rangle_{\mathrm{cpt}},\mathrm{h}*}
      \mathcal{I}_{\langle \Sigma \rangle,\mathrm{h}}\bigr)
    \end{tikzcd}
  \]
  and hence the maps on each direct summand are also surjective.
  By Excision \cite[Proposition 1.3]{Gro67}, the fact that 
  \[
    X_\mathrm{h} = \varprojlim_{\lambda \in \Lambda} X_{\lambda,\mathrm{h}}
  \]
  as sites by \cite[Lemma 3.6]{He24}, and the fact that the h
  topology on a scheme is compatible with localization,
  the surjective map on local cohomology with support in \(\fm\)
  can be identified with the
  surjection
  \[
    H^i_{\fm}(X,
    \mathcal{I}_{\Sigma}) \longtwoheadrightarrow
    H^i_{\fm}\bigl(X_{\mathrm{h}},
    \mathcal{I}_{\Sigma,\mathrm{h}}\bigr).
  \]
  Under the isomorphism
  \[
    \underline{\Omega}^0_{X,\Sigma,\mathrm{h}}
    \cong \mathcal{R}\rho_{X,\mathrm{h}*}\mathcal{I}_{\Sigma,\mathrm{h}},
  \]
  we obtain the desired surjection in Theorem \ref{thm:keyinjintro}.
\end{proof}

\subsection{A characterization of Du Bois pairs}
\par We can now show the following characterization of Du Bois pairs inspired by the
proofs of \cite[Theorem 5.4]{Kov11} and \cite[Theorem 2.5]{Kov12}.
See also \citeleft\citen{Kov99}\citemid Lemma 2.2\citepunct
\citen{Kov00}\citemid Corollary 1.5\citepunct
\citen{BST17}\citemid Theorem 4.8\citeright\ for related statements when \(\Sigma_X =
\emptyset\).
\begin{theorem}\label{thm:splittingcriterion}
  Let \(X\) be a separated Noetherian scheme of equal characteristic zero.
  Consider a closed subscheme \(\Sigma_X \subseteq X\).
  The following are equivalent:
  \begin{enumerate}[label=\((\roman*)\),ref=\roman*]
    \item\label{thm:splittingcriteriondb} \((X,\Sigma_X)\) has Du Bois
      singularities with respect to the h topology.
    \item\label{thm:splittingcriterionsplits}
      The natural morphism \(\mathcal{I}_{\Sigma_X} \to
      \underline{\Omega}^0_{X,\Sigma_X,\mathrm{h}}\) admits a left inverse in \(D(X)\).
    \item\label{thm:splittingcriterioninjonlc}
      \((X,\Sigma_X)\) is h-injective, that is,
      for every \(x \in X\), the natural morphism
      \begin{equation}\label{eq:naturalmaponlc}
        H^i_x\bigl(\Spec(\cO_{X,x}),\mathcal{I}_{\Sigma_X,x}\bigr)
        \longrightarrow \mathbb{H}^i_x\bigl(\Spec(\cO_{X,x}),
        \underline{\Omega}^0_{X,\Sigma_X,\mathrm{h},x}\bigr)
      \end{equation}
      is injective.
  \end{enumerate}
  If \(X\) is quasi-excellent, then these conditions are also equivalent to:
  \begin{enumerate}[label=\((\roman*)\),ref=\roman*,resume]
    \item\label{thm:splittingcriterioninjonlctau}
      \((X,\Sigma_X)\) is \(\tau\)-injective
      for some (resp.\ every) \(\tau \in
      \{\mathrm{cdh},\mathrm{eh},\mathrm{sdh},\mathrm{h}\}\).
  \end{enumerate}
\end{theorem}
\begin{proof}
  \((\ref{thm:splittingcriteriondb}) \Rightarrow
  (\ref{thm:splittingcriterionsplits})\) follows by the definition of a Du Bois
  pair, and \((\ref{thm:splittingcriterionsplits}) \Rightarrow
  (\ref{thm:splittingcriterioninjonlc})\) holds because the morphism
  \eqref{eq:naturalmaponlc} also admits a left inverse.\smallskip
  \par \((\ref{thm:splittingcriterioninjonlc}) \Rightarrow
  (\ref{thm:splittingcriteriondb})\).
  Suppose that \((X,\Sigma_X)\) is not Du Bois.
  Let \(x \in X\) be a point that is maximal among all points at which
  \((X,\Sigma_X)\) is not Du Bois.
  After localizing at \(x\), we may assume that \((X \setminus
  \{x\},\Sigma_X\rvert_{X \setminus \{x\}})\) is Du
  Bois and that \(X = \Spec(R)\) is the spectrum of a local ring \((R,\fm)\)
  where \(x = \fm\).
  By the hypothesis in \((\ref{thm:splittingcriterioninjonlc})\) and by Theorem
  \ref{thm:keyinj}, we know that \eqref{eq:naturalmaponlc}
  is an isomorphism for all \(i\).
  \par We now base change to the completion to show
  \((\ref{thm:splittingcriteriondb})\).
  This is necessary to ensure the existence of dualizing complexes, in which
  case we can apply local duality.
  Let
  \(\hat{\mathcal{I}}_{\Sigma_X} = \mathcal{I}_{\Sigma_X}
  \otimes_{R} \hat{R}\) and \(\hat{\Omega} =
  \underline{\Omega}^0_{X,\Sigma_X,\mathrm{h},x} \otimes_{R} \hat{R}\).
  Denote by \(\hat{\omega}^\bullet\) a dualizing complex on \(\hat{R}\).
  Applying Grothendieck local duality \cite[Chapter III, Theorem 6.2]{Har66} on \(\hat{R}\),
  we see that
  \[
    \operatorname{\mathcal{RH}\kern-1pt\mathit{om}}_{\hat{R}}(
    \hat{\Omega},\hat{\omega}^\bullet) \longrightarrow
    \operatorname{\mathcal{RH}\kern-1pt\mathit{om}}_X\bigl(
    \hat{\mathcal{I}}_{\Sigma_X},\hat{\omega}^\bullet\bigr)
  \]
  is a quasi-isomorphism.
  Since \(\hat{\omega}^\bullet\) is a dualizing complex and \(R \to \hat{R}\) is
  faithfully flat,
  we see that
  \(\mathcal{I}_{\Sigma_X} \to
  \underline{\Omega}^0_{X,\Sigma_X}\) is a quasi-isomorphism,
  that is, \((X,\Sigma_X)\) has Du Bois singularities.\smallskip
  \par Finally, when \(X\) is quasi-excellent,
  \((\ref{thm:splittingcriterioninjonlc})
  \Leftrightarrow (\ref{thm:splittingcriterioninjonlctau})\) by Remark
  \ref{rem:naturalisosehcdhrh} and Proposition
  \ref{prop:sameashyperresdef} (or by repeating the proof in the previous
  paragraph using the quasi-excellent case in Theorem
  \ref{thm:keyinj}).
\end{proof}
\section{Proofs of main theorems}\label{sect:mainresultproofs}
We prove a version of \citeleft\citen{Kov99}\citemid Corollary 2.4\citepunct
\citen{KK10}\citemid Theorem 1.6\citepunct
\citen{Kov12}\citemid Theorem 3.3\citeright\ 
that replaces splitting conditions with conditions on injectivity of maps on
local cohomology.
Note that \citeleft\citen{KK10}\citemid Theorem 1.6\citepunct
\citen{Kov12}\citemid Theorem 3.3\citeright\
assume that \(f\) is proper, which we do not need in the
statement below.
\begin{proposition}\label{thm:dbdescends}
  Let \(f\colon Y \to X\) be a morphism between separated Noetherian schemes.
  Let \(\tau \in
  \{\mathrm{cdh},\mathrm{eh},\mathrm{sdh},\mathrm{h}\}\).
  Consider a closed subscheme \(\Sigma_X \subseteq X\) and set
  \(\Sigma_{Y} \coloneqq f^{-1}(\Sigma_X)\).
  Assume one of the following conditions holds.
  \begin{enumerate}[label=\((\roman*)\),ref=\roman*]
    \item\label{thm:dbdescendsmaxpt}
      For every \(x \in X\), there exists a maximal point \(y \in
      f^{-1}(x)\) such that the natural morphism
      \[
        H^i_x\bigl(\Spec(\cO_{X,x}),\mathcal{I}_{\Sigma_X,x}\bigr)
        \longrightarrow H^i_y\bigl(\Spec(\cO_{Y,y}),
        \mathcal{I}_{\Sigma_Y,y}\bigr)
      \]
      is injective.
    \item\label{thm:dbdescendslocalize}
      \(Y\) is of equal characteristic zero, \(\tau = \mathrm{h}\),
      and for every \(x \in X\),
      the natural morphism
      \begin{equation}\label{eq:dbdescendsinjectivity}
        H^i_x\bigl(\Spec(\cO_{X,x}),\mathcal{I}_{\Sigma_X,x}\bigr)
        \longrightarrow \mathbb{H}^i_x\bigl(\Spec(\cO_{X,x}),
        (\mathcal{R} f_*\mathcal{I}_{\Sigma_{Y}})_x\bigr)
      \end{equation}
      is injective.
  \end{enumerate}
  If \((Y,\Sigma_{Y})\) is \(\tau\)-injective, then \((X,\Sigma_X)\) is
  \(\tau\)-injective.
  In particular, if \(Y\) is of equal characteristic zero and
  \((Y,\Sigma_Y)\) has Du Bois singularities with respect to the h topology, then:
  \begin{itemize}
    \item \((X,\Sigma_X)\) has Du Bois singularities with respect to the h topology; and
    \item \(X\) has Du Bois singularities with respect to the h topology
      if and only if \(\Sigma_X\) has Du Bois singularities with respect
      to the h topology.
  \end{itemize}
\end{proposition}
\begin{proof}
  Under each respective assumption,
  we have the commutative diagrams
  \[
    \begin{tikzcd}
      \mathbb{H}^i_x\bigl(\Spec(\cO_{X,x}),
      \underline{\Omega}^0_{X,\Sigma_X,\tau,x}\bigr)
      \rar & 
      \mathbb{H}^i_y\bigl(\Spec(\cO_{Y,y}),
      \underline{\Omega}^0_{Y,\Sigma_Y,\tau,y}\bigr)\\
      H^i_x\bigl(\Spec(\cO_{X,x}),\mathcal{I}_{\Sigma_X, x}\bigr)
      \rar[hook] \uar & 
      H^i_y\bigl(\Spec(\cO_{Y,y}),
      \mathcal{I}_{\Sigma_Y,y}\bigr)
      \arrow[hook]{u}
    \end{tikzcd}
  \]
  and
  \[
    \begin{tikzcd}
      \mathbb{H}^i_x\bigl(\Spec(\cO_{X,x}),
      \underline{\Omega}^0_{X,\Sigma_X,\mathrm{h},x}\bigr)
      \rar & 
      \mathbb{H}^i_x\Bigl(\Spec(\cO_{X,x}),
      \bigl(\mathcal{R} f_*\underline{\Omega}^0_{Y,\Sigma_{Y},\mathrm{h}}\bigr)_x\Bigr)\\
      H^i_x\bigl(\Spec(\cO_{X,x}),\mathcal{I}_{\Sigma_X, x}\bigr)
      \rar[hook] \uar & 
      \mathbb{H}^i_x\Bigl(\Spec(\cO_{X,x}),
      \bigl(\mathcal{R} f_*\mathcal{I}_{\Sigma_{Y}}\bigr)_x\Bigr)
      \arrow{u}[sloped,below]{\sim}
    \end{tikzcd}
  \]
  respectively
  for every \(i\), where the right vertical arrow is injective (resp.\ an
  isomorphism) by the assumption that \((Y,\Sigma_Y)\) is \(\tau\)-injective
  (resp.\ the assumption that \((Y,\Sigma_Y)\) is \(\tau\)-injective and by
  Theorem \ref{thm:splittingcriterion}), and the bottom
  horizontal arrow is injective by hypothesis.
  By the commutativity of the diagrams, we see that the left vertical arrow is
  injective for every \(i\).
  The statements for Du Bois singularities now follow from
  Theorem \ref{thm:splittingcriterion} and by considering the exact triangle of
  Du Bois defects (see \cite[Definition 3.11]{Kov11})
  \[
    \underline{\Omega}_{X,\Sigma_X}^\times \longrightarrow
    \underline{\Omega}_X^\times \longrightarrow
    \underline{\Omega}_{\Sigma_X}^\times \xrightarrow{+1}.\qedhere
  \]
\end{proof}
Using Proposition \ref{thm:dbdescends},
we can prove a more general version of Theorem \ref{thm:boutotdb} for pairs.
The statement \((\ref{cor:dbdescendssplit})\) below is a version of
\citeleft\citen{Kov99}\citemid Corollary 2.4\citepunct
\citen{KK10}\citemid Theorem 1.6\citepunct
\citen{Kov12}\citemid Theorem 3.3\citeright.
\par By \cite[Theorem 4.8]{BST17} and its proof (see
\citeleft\citen{BST17}\citemid Theorem 3.3\citepunct
\citen{BS17}\citemid Theorem 4.1\((i)\)\citepunct
\citen{stacks-project}\citemid
\href{https://stacks.math.columbia.edu/tag/0EVW}{Tag
0EVW}\citeright), statements \((\ref{cor:dbdescendsfflat})\)
and \((\ref{cor:yamstar})\) give a new proof that \(F\)-injectivity descends
under faithfully flat maps or maps that localize to pure maps
\cite[Theorem 3.8 and Proposition 3.9]{DM}.
We will use the following terminology in statement
\((\ref{cor:dbdescendspartiallypure})\) below.
\begin{definition}[cf.\ {\cite[p.\ 38]{CGM16}}]
  Let \(f\colon Y \to X\) be a morphism of locally Noetherian schemes.
  For a point \(x \in X\), we say that \(f\) is \textsl{partially pure at \(x\)}
  if there exists a point \(y \in f^{-1}(x)\) such that \(\cO_{X,x} \to
  \cO_{Y,y}\) is pure.
\end{definition}
\begin{theorem}\label{cor:dbdescends}
  Let \(f\colon Y \to X\) be a surjective morphism between Noetherian schemes.
  Let \(\tau \in
  \{\mathrm{rh},\mathrm{cdh},\mathrm{eh},\mathrm{sdh},\mathrm{h}\}\).
  Consider a closed subscheme \(\Sigma_X \subseteq X\) and set
  \(\Sigma_{Y} \coloneqq f^{-1}(\Sigma_X)\).
  Assume one of the following conditions holds.
  \begin{enumerate}[label=\((\roman*)\),ref=\roman*]
    \item\label{cor:dbdescendsfflat} \(f\) is faithfully flat.
    \item\label{cor:yamstar}
      \(\Sigma_X = \emptyset\) and for every \(x \in X\), there is a maximal
      point \(y \in f^{-1}(x)\) such that
      \(\cO_{X,x} \to \cO_{Y, y}\) is pure (cf.\ \emph{\cite[Definition
      5.7]{Yam}}).
    \item\label{cor:dbdescendssplit}
      \(Y\) is of equal characteristic zero, \(\tau = \mathrm{h}\), and
      the natural morphism \(\mathcal{I}_{\Sigma_X} \to \mathcal{R}
      f_*\mathcal{I}_{\Sigma_Y}\) admits a left inverse in \(D(X)\).
    \item\label{cor:dbdescendspure}
      \(Y\) is of equal characteristic zero, \(\tau = \mathrm{h}\),
      \(f\) is affine, and for every affine open subset \(U \subseteq X\),
      the \(H^0(U,\cO_X)\)-module map
      \[
        H^0\bigl(U,\mathcal{I}_{\Sigma_X}\bigr) \longrightarrow
        H^0\bigl(f^{-1}(U),\mathcal{I}_{\Sigma_Y}\bigr)
      \]
      is pure.
    \item\label{cor:dbdescendspartiallypure}
      \(Y\) is of equal characteristic zero, \(\tau = \mathrm{h}\),
      \(\Sigma_X = \emptyset\), and
      \(f\) is partially pure at every \(x \in X\).
  \end{enumerate}
  If \((Y,\Sigma_{Y})\) is \(\tau\)-injective, then \((X,\Sigma_X)\) is
  \(\tau\)-injective.
  In particular, if \((Y,\Sigma_Y)\) is of equal characteristic zero and has Du Bois
  singularities with respect to the h topology and one of the conditions
  \((\ref{cor:dbdescendsfflat})\)--\kern1pt\((\ref{cor:dbdescendspartiallypure})\) holds, then:
  \begin{itemize}
    \item \((X,\Sigma_X)\) has Du Bois singularities with respect to the h topology; and
    \item \(X\) has Du Bois singularities with respect to the h topology
      if and only if \(\Sigma_X\) has Du Bois singularities with respect
      to the h topology.
  \end{itemize}
\end{theorem}
\begin{proof}
  \par For \((\ref{cor:dbdescendsfflat})\) and
  \((\ref{cor:yamstar})\), we show that the hypothesis in Proposition
  \ref{thm:dbdescends}\((\ref{thm:dbdescendsmaxpt})\) is satisfied.
  For \((\ref{cor:dbdescendsfflat})\), let \(x \in X\) be a point
  and let \(y \in Y\) such that \(f(y) = x\).
  The morphism \(\Spec(\cO_{Y,y}) \to \Spec(\cO_{X,x})\) is faithfully flat, and
  hence
  \[
    \mathcal{I}_{\Sigma_X,x} \longrightarrow \mathcal{I}_{\Sigma_X,x}
    \otimes_{\cO_{X,x}} \cO_{Y,y} \simeq \mathcal{I}_{\Sigma_Y,y}
  \]
  is pure by 
  \cite[p.\ 136]{HR74}.
  For \((\ref{cor:yamstar})\), we choose \(y\) as in the assumption.
  In either situation,
  the associated map on local cohomology is injective by
  \cite[Corollary 3.2\((a)\)]{Kem79}.\smallskip
  \par For \((\ref{cor:dbdescendssplit})\) and \((\ref{cor:dbdescendspure})\), we
  show that the hypothesis in Proposition
  \ref{thm:dbdescends}\((\ref{thm:dbdescendslocalize})\) is satisfied.
  For \((\ref{cor:dbdescendssplit})\), it suffices to note that the morphism
  \eqref{eq:dbdescendsinjectivity} admits a left inverse.
  For \((\ref{cor:dbdescendspure})\), the morphism
  \eqref{eq:dbdescendsinjectivity} is
  \[
    H^i_x\Bigl(\Spec(\cO_{X,x}),H^0\bigl(U,\mathcal{I}_{\Sigma_X}\bigr)_x\Bigr)
    \longrightarrow H^i_x\Bigl(\Spec(\cO_{X,x}),H^0\bigl(f^{-1}(U),
    \mathcal{I}_{\Sigma_Y}\bigr)_x\Bigr)
  \]
  where \(U\) is an affine open containing \(x\).
  This map is injective by \cite[Corollary 3.2\((a)\)]{Kem79}.\smallskip
  \par For \((\ref{cor:dbdescendspartiallypure})\), let \(x \in X\) be a point.
  By assumption, there exists
  \(y \in Y\) such that \(f(y) = x\) and
  the map \(\cO_{X,x} \to \cO_{Y,y}\) is pure,
  in which case \((\ref{cor:dbdescendspure})\) applies to the morphism
  \(\Spec(\cO_{Y,y}) \to \Spec(\cO_{X,x})\).
\end{proof}
We single out a special case of Theorem \ref{cor:dbdescends}, which generalizes
\cite[Corollary 2.5]{Kov99}.
See \citeleft\citen{EGAIV3}\citemid D\'efinition 13.2.2 and D\'efinition
13.3.2\citepunct \citen{EGAIV4}\citemid Err\textsubscript{IV}, 34 and
35\citeright\ for the definition of locally
equidimensional morphisms.
Finite dominant morphisms of integral schemes
are locally equidimensional.

\begin{corollary}\label{cor:finitesurjdb}
  Let \(f\colon Y \to X\) be a locally equidimensional
  surjective morphism between integral
  Noetherian schemes of equal characteristic zero.
  Suppose that \(Y\) has Du Bois singularities with respect to the h
  topology and that \(X\) is normal.
  Then, \(X\) has Du Bois singularities with respect to the h
  topology.
\end{corollary}
\begin{proof}
  Let \(x \in X\) be a point.
  Replacing \(X\) by an affine open neighborhood of \(x\) and shrinking \(Y\),
  we may assume that
  \(X = \Spec(R)\) and \(Y = \Spec(S)\) are affine.
  Since \(R\) is a normal \(\QQ\)-algebra, the inclusion \(R \hookrightarrow S\)
  is pure \cite[Proof of Corollary A.2]{Zhuang}.
  We can now apply Theorem \ref{cor:dbdescends}\((\ref{cor:dbdescendssplit})\).
\end{proof}
Finally, we show Theorem \ref{thm:boutotdb} and Corollary \ref{cor:boutotlcqgor}.
\begin{proof}[Proof of Theorem \ref{thm:boutotdb}]
  Set
  \[
    (X,\Sigma_X) = \bigl(\Spec(R),\emptyset\bigr) \qquad \text{and} \qquad
    (Y,\Sigma_Y) = \bigl(\Spec(S),\emptyset\bigr).
  \]
  Since \(R \to S\) is cyclically pure, the map \(Y \to X\) is surjective.
  Next, \(S\) is reduced since it has Du Bois singularities with respect to the
  h topology, and hence the
  subring \(R\) of \(S\) is reduced.
  As a consequence, \(R \to S\) is pure by \cite[Proposition 1.4 and Theorem
  1.7]{Hoc77}.
  Thus, Theorem \ref{cor:dbdescends}\((\ref{cor:dbdescendspure})\) applies.
\end{proof}
\begin{proof}[Proof of Corollary \ref{cor:boutotlcqgor}]
  For \((\ref{cor:boutotlcqgorcartier})\), note that if \(S\) has log canonical
  type singularities, then it is normal (by
  definition) and Du Bois (by \cite[Theorem 1.4]{KK10}).
  It therefore suffices to consider the case when \(S\) is normal and Du
  Bois.
  In this case, \(R\) is normal (by \cite[Proposition 6.15\((b)\)]{HR76}) and Du Bois
  (by Theorem \ref{thm:boutotdb}).
  Finally, since \(K_R\) is Cartier, we see that \(R\) being Du Bois implies
  that \(R\) has log canonical singularities \cite[Theorem
  3.3]{Kov99}.\smallskip
  \par For \((\ref{cor:boutotlcqgorqcartier})\), we note that \(R\) is reduced
  since it is a subring of \(S\), and hence \(R\to S\) is pure
  by \cite[Proposition 1.4 and Theorem 1.7]{Hoc77}.
  Denote by \(f\colon Y \to X\)
  the morphism of affine schemes associated to \(R \to S\).
  We adapt the proof of \cite[Lemma 2.3]{Zhuang}.
  We want to show that \(X\) has log canonical singularities at every closed
  point \(x \in X\).
  Since the statement is local at \(x\), we can replace \(X\) with an affine
  open neighborhood of \(x\) to assume that \(\cO_X(rK_X) \cong \cO_X\), where
  \(r\) is the Cartier index of \(K_X\) at \(x\).
  \par Let \(s\in H^0(X,\cO_X(rK_X))\) be a nowhere vanishing section, and let
  \(\pi\colon X' \to X\) be the corresponding index 1 cover as in
  \cite[Definition 5.19]{KM98}.
  Let \(Y'\) be the normalization of the components of \(X' \times_X Y\)
  dominating \(X'\), and denote by \(\pi'\) the composition \(Y' \to X' \times_X Y
  \to Y\).
  Let \(U\) be the Cartier locus of \(K_X\), let \(V = f^{-1}(U)\), and let \(U'\)
  (resp.\ \(V'\)) be the preimage of \(U\) (resp.\ \(V\)) in \(X'\) (resp.\
  \(Y'\)).
  Then, \(\pi\) and \(\pi'\) are \'etale over \(U\) and \(V\), respectively
  (see \cite[Definition 2.49]{KM98}), and
  \begin{align*}
    \pi_*\cO_{U'} &= \left. \Biggl( \bigoplus_{m \in \NN} \cO_U(mK_U) \cdot
    t^m\Biggr)\middle/ (st^r-1)\right.\\
    \pi'_*\cO_{V'} &= \left. \Biggl( \bigoplus_{m \in \NN} \cO_V(mf^*K_U) \cdot
    t^m\Biggr)\middle/ (f^*s\,t^r-1)\right.
  \end{align*}
  Here, in the second equality, \(f^*s\) denotes the pullback of \(s\) to
  \(H^0(V,\cO_V(rf^*K_U))\).
  By assumption, the complement of \(U\) (resp.\ \(V\)) in \(X\) (resp.\ \(Y\))
  has codimension at least two.
  We therefore have
  \begin{alignat*}{3}
    H^0(X',\cO_{X'}) &= H^0(U,\pi_*\cO_{U'})
    &{}=& \left. \Biggl( \bigoplus_{m \in \NN} H^0\bigl(U,\cO_U(mK_U)\bigr) \cdot
    t^m\Biggr)\middle/ (st^r-1)\right.\\
    H^0(Y',\cO_{Y'}) &= H^0(V,\pi'_*\cO_{V'})
    &{}=& \left. \Biggl( \bigoplus_{m \in \NN} H^0\bigl(V,\cO_V(mf^*K_U)\bigr) \cdot
    t^m\Biggr)\middle/ (f^*s\,t^r-1)\right.
  \end{alignat*}
  By \cite[Lemma 2.2]{Zhuang}, the ring map \(H^0(X',\cO_{X'}) \to
  H^0(Y',\cO_{Y'})\) is pure.
  Since \(Y' \to Y\) is \'etale in codimension one (it is \'etale over \(V\)),
  we know \(Y'\) has log canonical type singularities \cite[Proposition 5.20]{KM98}.
  By \((\ref{cor:boutotlcqgorcartier})\), this implies \(X'\) has log
  canonical singularities.
  Finally, since \(X' \to X\) is \'etale in codimension one (it is \'etale over
  \(U\)), we see that \(X\) has log canonical singularities \cite[Proposition
  5.20]{KM98}.
\end{proof}

\end{document}